\documentclass[11pt]{amsart}
 
\usepackage{fullpage}
\usepackage{amssymb}
\usepackage{amsmath}
\usepackage{amsxtra}
\usepackage{amscd}
\usepackage{graphicx}
\usepackage{amsfonts}
\usepackage{float}
\usepackage{enumerate}
\usepackage{color}

\newtheorem{theorem}{Theorem}
\newtheorem*{theorem*}{Theorem}
\newtheorem{lemma}{Lemma}

\newtheorem{prop}[]{Proposition}

\newtheorem{definition}{Definition}

\usepackage{graphicx,color}
\usepackage{pstricks-add}
\usepackage{pgf,tikz,pgfplots}
\usepackage{mathrsfs}
\usetikzlibrary{arrows}
\pgfplotsset{compat=1.15}

\numberwithin{equation}{section}

\begin{document}

\title[]
  {Knot invariants for rail knotoids}
\author{Dimitrios Kodokostas}
\address{Department of Mathematics,
National Technical University of Athens,
Zografou campus,
{GR-15780} Athens, Greece.}
\email{dkodokostas@math.ntua.gr}

\author{Sofia Lambropoulou}
\address{Department of Mathematics,
National Technical University of Athens,
Zografou campus,
{GR-15780} Athens, Greece.}
\email{sofia@math.ntua.gr}
\urladdr{http://www.math.ntua.gr/$\sim$sofia}

\subjclass[2010]{57K10, 57K14}
\date{}

\keywords{rail arc, rail isotopy, rail knotoid diagram, rail equivalence, rail knotoid, knot, knot invariant}

\begin{abstract} To each rail knotoid we associate two unoriented knots along with their oriented counterparts, thus deriving invariants for rail knotoids based on these associations. We then translate them to invariants of rail isotopy for rail arcs.   
\end{abstract}

\dedicatory{In memory of V.F.R. Jones}

\maketitle

\section*{Introduction} \label{section_introduction}

In \cite{KL} we studied isotopies in $\mathbb{R}^3$ between rail arcs  which are just arcs with endpoints on two fixed parallel lines and which have no other point on these lines. Rail arcs were introduced in \cite{GK}.  During the isotopies we allow the endpoints to move freely on the two lines, the rails. The way to study these isotopies was by studying the projection of the arcs on the plane of the lines and introducing the notion of rail knotoid diagrams and their equivalence, proving that rail arc isotopy in space is equivalent to rail knotoid diagram equivalence on the plane. We provide a small summary of all these in the next paragraph. In \S 3 we provide a way of distinguishing rail knotoids by defining analogues for rail knotoids of the normalized bracket, the Jones, the HOMFLYPT and the Kauffman polynomials. We achieve this by assigning in \S 2 to each rail knotoid two well-defined isotopy classes of knots. In \S 4 we translate these invariants to invariants of space isotopy of rail arcs. In the last paragraph we comment on the connection of the rail knotoids to the usual knotoids which were first introduced in \cite{Tu}.

\section{A summary on rail arcs and rail knotoids} \label{section_rail_knotoids}

We consider the space $\mathbb{R}^3$ equipped with two parallel lines $\ell_1,\ell_2$. We call as rail arc any connected, embedded simple arc $c$ in $\mathbb{R}^3$ with one endpoint on $\ell_1$ and the other on $\ell_2$ and otherwise missing the rails at all (see Figure \ref{figure_a_rail_arc}). If we call $\ell_1,\ell_2$ as first and second rail respectivey, then we can consider a natural orientation on any such arc $c$ from the endpoint on $\ell_1$ to the endpoint on $\ell_2$, in which case we call the enpoints as leg and head respectively. Although matters of orientation for rail arcs will not be of out-most importance in what follows, we shall keep this terminology in all cases whenever referring to the endpoints.

\begin{figure}[!h]
	\centering
\definecolor{ffqqqq}{rgb}{1.,0.,0.}
\begin{tikzpicture}[line cap=round,line join=round,>=triangle 45,x=1.0cm,y=1.0cm]
\clip(3.5065194234842947,4.4) rectangle (6.668613714896429,7.670929855692563);
\draw [shift={(6.165968007865688,6.320138130917781)},line width=0.4pt]  plot[domain=-0.9090662575001538:0.9923514156973319,variable=\t]({1.*0.24226291935302283*cos(\t r)+0.*0.24226291935302283*sin(\t r)},{0.*0.24226291935302283*cos(\t r)+1.*0.24226291935302283*sin(\t r)});
\draw [shift={(5.219811882128152,4.820081265772829)},line width=0.4pt]  plot[domain=1.2981046296861078:1.718995016247648,variable=\t]({1.*2.023928926136264*cos(\t r)+0.*2.023928926136264*sin(\t r)},{0.*2.023928926136264*cos(\t r)+1.*2.023928926136264*sin(\t r)});
\draw [shift={(6.204868053391834,6.34658325187964)},line width=0.4pt]  plot[domain=1.8031194837987232:3.0144356198078617,variable=\t]({1.*0.48491668043597475*cos(\t r)+0.*0.48491668043597475*sin(\t r)},{0.*0.48491668043597475*cos(\t r)+1.*0.48491668043597475*sin(\t r)});
\draw [shift={(6.309704444107993,6.62558684206024)},line width=0.4pt]  plot[domain=-0.8893986781627756:0.8893986781627703,variable=\t]({1.*0.21663861102863105*cos(\t r)+0.*0.21663861102863105*sin(\t r)},{0.*0.21663861102863105*cos(\t r)+1.*0.21663861102863105*sin(\t r)});
\draw [shift={(6.02638382590094,6.382174114055326)},line width=0.4pt]  plot[domain=3.0561739652779294:4.962142713419938,variable=\t]({1.*0.3036244470810311*cos(\t r)+0.*0.3036244470810311*sin(\t r)},{0.*0.3036244470810311*cos(\t r)+1.*0.3036244470810311*sin(\t r)});
\draw [shift={(5.913020909225199,6.165646600533312)},line width=0.4pt]  plot[domain=0.7476423546557244:1.4462476259445392,variable=\t]({1.*0.5255705655408514*cos(\t r)+0.*0.5255705655408514*sin(\t r)},{0.*0.5255705655408514*cos(\t r)+1.*0.5255705655408514*sin(\t r)});
\draw [shift={(6.234937330459185,6.3080246633316115)},line width=0.4pt]  plot[domain=1.1606768062067996:1.8416066999258576,variable=\t]({1.*0.5297547816122091*cos(\t r)+0.*0.5297547816122091*sin(\t r)},{0.*0.5297547816122091*cos(\t r)+1.*0.5297547816122091*sin(\t r)});
\draw [shift={(4.6292332883388685,5.861498067222166)},line width=0.4pt]  plot[domain=4.968975415392938:6.257436701285708,variable=\t]({1.*0.5007769968756877*cos(\t r)+0.*0.5007769968756877*sin(\t r)},{0.*0.5007769968756877*cos(\t r)+1.*0.5007769968756877*sin(\t r)});
\draw [shift={(4.796311533076281,5.838184648633326)},line width=0.4pt]  plot[domain=0.031232750135664213:1.3166279585845444,variable=\t]({1.*0.3336955002566434*cos(\t r)+0.*0.3336955002566434*sin(\t r)},{0.*0.3336955002566434*cos(\t r)+1.*0.3336955002566434*sin(\t r)});
\draw [shift={(4.804753146134957,5.854290983986243)},line width=0.4pt]  plot[domain=3.253165795916294:5.163232756078342,variable=\t]({1.*0.32831994734181613*cos(\t r)+0.*0.32831994734181613*sin(\t r)},{0.*0.32831994734181613*cos(\t r)+1.*0.32831994734181613*sin(\t r)});
\draw [shift={(5.337200378913295,6.745049224444948)},line width=0.4pt]  plot[domain=4.9134570762437075:5.832726076495774,variable=\t]({1.*1.1133137627604077*cos(\t r)+0.*1.1133137627604077*sin(\t r)},{0.*1.1133137627604077*cos(\t r)+1.*1.1133137627604077*sin(\t r)});
\draw [shift={(4.642083259921757,9.789377117550492)},line width=0.4pt]  plot[domain=4.827682869882762:4.930718586642295,variable=\t]({1.*4.235767135465101*cos(\t r)+0.*4.235767135465101*sin(\t r)},{0.*4.235767135465101*cos(\t r)+1.*4.235767135465101*sin(\t r)});
\draw (3.66203225748817,4.910577052123722) node[anchor=north west] {$\ell_1$};
\draw (4.795980005433095,4.871698843622752) node[anchor=north west] {$\ell_2$};
\draw (5.281957611695206,7.152553742346302) node[anchor=north west] {$c$};
\draw [shift={(5.115547940464045,5.172922045256557)},line width=0.4pt]  plot[domain=1.8407341451325392:2.3416789290169016,variable=\t]({1.*1.639697913995095*cos(\t r)+0.*1.639697913995095*sin(\t r)},{0.*1.639697913995095*cos(\t r)+1.*1.639697913995095*sin(\t r)});
\draw [shift={(4.154020345409349,5.82935154278977)},line width=0.4pt]  plot[domain=2.495303727797669:4.415958829430152,variable=\t]({1.*0.4864099869394543*cos(\t r)+0.*0.4864099869394543*sin(\t r)},{0.*0.4864099869394543*cos(\t r)+1.*0.4864099869394543*sin(\t r)});
\draw [shift={(4.198744547184168,5.975591806272608)},line width=0.4pt]  plot[domain=4.415871813940938:5.489248162548796,variable=\t]({1.*0.6393363304296511*cos(\t r)+0.*0.6393363304296511*sin(\t r)},{0.*0.6393363304296511*cos(\t r)+1.*0.6393363304296511*sin(\t r)});
\draw [shift={(4.4655152331748535,5.704340475685186)},line width=0.4pt]  plot[domain=-0.35230872467414187:1.1133214788988537,variable=\t]({1.*0.2347271279265541*cos(\t r)+0.*0.2347271279265541*sin(\t r)},{0.*0.2347271279265541*cos(\t r)+1.*0.2347271279265541*sin(\t r)});
\draw [shift={(4.186342079628072,5.1096278206348416)},line width=0.4pt]  plot[domain=1.1270142290891665:1.937138884575772,variable=\t]({1.*0.8916757943429974*cos(\t r)+0.*0.8916757943429974*sin(\t r)},{0.*0.8916757943429974*cos(\t r)+1.*0.8916757943429974*sin(\t r)});
\draw [shift={(4.681206570762683,5.963942130321101)},line width=0.4pt]  plot[domain=1.5145815863593217:2.6106294192079087,variable=\t]({1.*0.19752931897705056*cos(\t r)+0.*0.19752931897705056*sin(\t r)},{0.*0.19752931897705056*cos(\t r)+1.*0.19752931897705056*sin(\t r)});
\draw [line width=1.2pt,color=ffqqqq] (3.866941057018287,7.619092244357937)-- (3.866941057018287,5.565026895223377);
\draw [line width=1.2pt,color=ffqqqq] (3.866941057018287,5.318798241383902)-- (3.866941057018287,4.826340933704954);
\draw [line width=1.2pt,color=ffqqqq] (4.756320587529386,7.619092244357937)-- (4.756320587529386,5.616864506558003);
\draw [line width=1.2pt,color=ffqqqq] (4.756320587529386,4.826340933704954)-- (4.756320587529386,5.467831373970953);
\begin{scriptsize}
\draw [fill=black] (4.756320587529386,5.377115554135357) circle (1.0pt);
\draw [fill=black] (3.866941057018287,5.942135303467903) circle (1.0pt);
\end{scriptsize}
\end{tikzpicture}
	\caption{A rail arc $c$ in $\mathbb{R}^3$ with its endpoints on the rails $\ell_1,\ell_2$.}
\label{figure_a_rail_arc}
\end{figure}
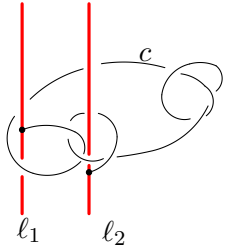

 We call two rail arcs  $c_1,c_2$  as rail isotopic, whenever there exists an isotopy  of  $\mathbb{R}^3$ taking one onto the other so that each rail maps onto itself (but not necessarily pointwise) throughout the isotopy.  In particular, this implies that at each time throughout the isotopy, the image of the arc is a rail arc, and each endpoint remains on the same rail with the freedom to move up and down on it. We call such an isotopy as a rail isotopy in $\mathbb{R}^3$.

\begin{figure}[!h]
	\centering
	\definecolor{ffqqqq}{rgb}{1.,0.,0.}
	\begin{tikzpicture}[line cap=round,line join=round,>=triangle 45,x=1.0cm,y=1.0cm]
	\clip(-0.2,0.11554440862388769) rectangle (6.7,3.9224116121245562);
	\draw (3.1590831577667435,2.301773859777129) node[anchor=north west] {$\sim$};
	\draw [line width=0.4pt,dash pattern=on 1pt off 1pt] (0.6445315826999247,1.6964695113948716)-- (1.197996206547054,2.608469175390208);
	\draw [line width=0.4pt,dash pattern=on 1pt off 1pt] (1.197996206547054,2.608469175390208)-- (2.00322258244179,1.6964695113948716);
	\draw [line width=0.4pt] (0.6445315826999247,1.6964695113948716)-- (2.00322258244179,1.6964695113948716);
	\draw [line width=0.4pt] (0.6445315826999247,1.6964695113948716)-- (1.1903889753849768,1.2576045696740883);
	\draw [line width=0.4pt] (2.00322258244179,1.6964695113948716)-- (2.2979448314727335,1.8773675685835884);
	\draw [line width=0.4pt] (2.2979448314727335,1.8773675685835884)-- (1.9586628065743275,2.216649593481997);
	\draw [line width=0.4pt] (0.48114431104901,2.457430385345384)-- (0.15709644872051595,0.7463966880611413);
	\draw [line width=0.4pt,color=ffqqqq] (1.7507157590559494,3.8583368107323617)-- (1.7507157590559494,2.332738656267294);
	\draw [line width=0.4pt,color=ffqqqq] (1.7507157590559494,2.1947604305853257)-- (1.7507157590559494,2.05248087175696);
	\draw [line width=0.4pt,color=ffqqqq] (1.7507157590559494,1.6146976138235303)-- (1.7507157590559494,0.5530732133349612);
	\draw [line width=0.4pt,color=ffqqqq] (1.7507157590559494,1.9119991082678132)-- (1.7507157590559494,1.7552530021503596);
	\draw [line width=0.4pt] (0.48114431104901,2.457430385345384)-- (0.9528997528226542,2.380551720760049);
	\draw [line width=0.4pt] (1.1054122664252921,2.3556978296544338)-- (1.3771607537773767,2.311412890974834);
	\draw [line width=0.4pt] (1.567414459270753,2.2804085834129504)-- (1.9586628065743275,2.216649593481997);
	\draw [line width=0.4pt,color=ffqqqq] (2.3745569016110837,3.8583368107323617)-- (2.3745569016110837,0.5530732133349612);
	\draw [line width=0.4pt] (4.581919947463449,1.6964695113948716)-- (5.135384571310579,2.608469175390208);
	\draw [line width=0.4pt] (5.135384571310579,2.608469175390208)-- (5.940610947205315,1.6964695113948716);
	\draw [line width=0.4pt] (4.581919947463449,1.6964695113948716)-- (5.127777340148501,1.2576045696740883);
	\draw [line width=0.4pt] (5.940610947205315,1.6964695113948716)-- (6.235333196236258,1.8773675685835884);
	\draw [line width=0.4pt] (6.235333196236258,1.8773675685835884)-- (5.896051171337852,2.216649593481997);
	\draw [line width=0.4pt] (4.418532675812535,2.457430385345384)-- (4.094484813484041,0.7463966880611413);
	\draw [line width=0.4pt,color=ffqqqq] (5.688104123819474,3.8583368107323617)-- (5.688104123819474,2.332738656267294);
	\draw [line width=0.4pt,color=ffqqqq] (5.688104123819474,2.1947604305853257)-- (5.688104123819474,2.05248087175696);
	\draw [line width=0.4pt] (4.418532675812535,2.457430385345384)-- (4.890288117586179,2.380551720760049);
	\draw [line width=0.4pt] (5.0428006311888165,2.3556978296544338)-- (5.314549118540901,2.311412890974834);
	\draw [line width=0.4pt] (5.504802824034278,2.2804085834129504)-- (5.896051171337852,2.216649593481997);
	\draw [line width=0.4pt,color=ffqqqq] (6.311945266374608,3.8583368107323617)-- (6.311945266374608,0.5530732133349612);
	\draw [line width=0.4pt,color=ffqqqq] (5.688104123819474,1.9119991082678132)-- (5.688104123819474,0.5530732133349612);
	\draw (1.6363362763664835,0.6049987633596879) node[anchor=north west] {$\ell_1$};
	\draw (2.278065319242307,0.6267522902368347) node[anchor=north west] {$\ell_2$};
	\draw (5.628108458322879,0.670259343991128) node[anchor=north west] {$\ell_1$};
	\draw (6.226330447444409,0.6485058171139813) node[anchor=north west] {$\ell_2$};
	\draw [line width=0.4pt] (0.15709644872051595,0.7463966880611413)-- (1.6643261105468938,0.9090436340637899);
	\draw [line width=0.4pt] (1.8494231157593275,0.929017672111554)-- (2.3745569016110837,0.9856854837097546);
	\draw [line width=0.4pt] (4.094484813484041,0.7463966880611413)-- (5.601714475310418,0.9090436340637899);
	\draw [line width=0.4pt] (5.786811480522852,0.929017672111554)-- (6.311945266374608,0.9856854837097546);
	\small
	\draw (0.35,1.67) node[anchor=north west] {$A$};
	\draw (4.3,1.75) node[anchor=north west] {$A$};
	\draw (1.9,1.82) node[anchor=north west] {$B$};
	\draw (5.8,1.7688124512870351) node[anchor=north west] {$B$};
	\draw (4.964625888569908,3.0413937736001158) node[anchor=north west] {$C$};
	\normalsize
	\draw [line width=0.4pt] (1.1903889753849768,1.2576045696740883)-- (1.7507157590559494,1.377248967498395);
	\draw [line width=0.4pt] (5.127777340148501,1.2576045696740883)-- (5.688104123819474,1.377248967498395);
	\draw (-0.1,1.76) node[anchor=north west] {$c$};
	\draw (3.8,1.76) node[anchor=north west] {$c'$};
	\begin{scriptsize}
	\draw [fill=black] (0.6445315826999247,1.6964695113948716) circle (0.5pt);
	\draw [fill=black] (2.00322258244179,1.6964695113948716) circle (0.5pt);
	\draw [fill=black] (1.1903889753849768,1.2576045696740883) circle (0.5pt);
	\draw [fill=black] (2.2979448314727335,1.8773675685835884) circle (0.5pt);
	\draw [fill=black] (1.9586628065743275,2.216649593481997) circle (0.5pt);
	\draw [fill=black] (0.48114431104901,2.457430385345384) circle (0.5pt);
	\draw [fill=black] (0.15709644872051595,0.7463966880611413) circle (0.5pt);
	\draw [fill=black] (4.581919947463449,1.6964695113948716) circle (0.5pt);
	\draw [fill=black] (5.135384571310579,2.608469175390208) circle (0.5pt);
	\draw [fill=black] (5.940610947205315,1.6964695113948716) circle (0.5pt);
	\draw [fill=black] (4.581919947463449,1.6964695113948716) circle (0.5pt);
	\draw [fill=black] (5.135384571310579,2.608469175390208) circle (0.5pt);
	\draw [fill=black] (5.135384571310579,2.608469175390208) circle (0.5pt);
	\draw [fill=black] (5.940610947205315,1.6964695113948716) circle (0.5pt);
	\draw [fill=black] (4.581919947463449,1.6964695113948716) circle (0.5pt);
	\draw [fill=black] (5.940610947205315,1.6964695113948716) circle (0.5pt);
	\draw [fill=black] (4.581919947463449,1.6964695113948716) circle (0.5pt);
	\draw [fill=black] (5.127777340148501,1.2576045696740883) circle (0.5pt);
	\draw [fill=black] (5.940610947205315,1.6964695113948716) circle (0.5pt);
	\draw [fill=black] (6.235333196236258,1.8773675685835884) circle (0.5pt);
	\draw [fill=black] (6.235333196236258,1.8773675685835884) circle (0.5pt);
	\draw [fill=black] (5.896051171337852,2.216649593481997) circle (0.5pt);
	\draw [fill=black] (4.418532675812535,2.457430385345384) circle (0.5pt);
	\draw [fill=black] (4.094484813484041,0.7463966880611413) circle (0.5pt);
	\draw [fill=black] (4.418532675812535,2.457430385345384) circle (0.5pt);
	\draw [fill=black] (5.896051171337852,2.216649593481997) circle (0.5pt);
	\draw [fill=black] (5.127777340148501,1.2576045696740883) circle (0.5pt);
	\draw [fill=black] (2.3745569016110837,0.9856854837097546) circle (0.5pt);
	\draw [fill=black] (1.7507157590559494,1.377248967498395) circle (0.5pt);
	\draw [fill=black] (4.094484813484041,0.7463966880611413) circle (0.5pt);
	\draw [fill=black] (6.311945266374608,0.9856854837097546) circle (0.5pt);
	\draw [fill=black] (5.127777340148501,1.2576045696740883) circle (0.5pt);
	\draw [fill=black] (5.688104123819474,1.377248967498395) circle (0.5pt);
	\end{scriptsize}
	\end{tikzpicture}
	\caption{A triangle move between rail arcs $c,c'$ in $\mathbb{R}^3$.}
	\label{figure_triangle_moves}
\end{figure}
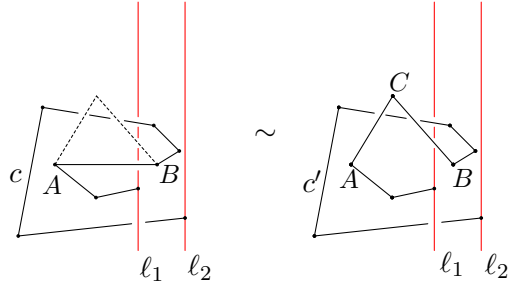

Rail isotopy between p.l. rail arcs can be effected via a finite sequence of triangle moves in space as  on Figure \ref{figure_triangle_moves}: a rail arc is modified so that it either replaces an edge $AB$ by two new edges $AC,CB$, or vice versa, where the triangle $ABC$ of the move does not intersect the arc or the rails at any other point. We also allow modifications via moves which slide the initial or final point of the arc on its rail, which means we allow replacement of an edge $MX$ with an edge $MX'$, where $X,X'$ lie on the same rail and the triangle $MXX'$ has no other common points with the rail other than the segment $XX'$. Here we think of the arcs and maps as piecewise linear  but due to the usual p.l. approximation theorems for the analogous smooth objects, our results hold in the smooth category as well.

It is natural to expect that rail isotopy can be described by some kind of equivalence among modifications of the projections of the rail arcs on a fixed plane. To this end, we choose to project the rail arcs on the plane $\pi$ defined by the rails $\ell_1,\ell_2$, assuming without any loss of generality that the arcs are in a generic position with respect to $\pi$. For any such projection $c_{pr}$ of some rail arc $c$, we keep track of the over/under data at double points of the projection with itself and with the rails. In this way we get what we call a rail knotoid diagram $c_{pr}$ on $\pi$, whose endpoints are on the rails (the leg on $\ell_1$ and the head on $\ell_2$). In general, we call as a planar rail knotoid diagram or just rail knotoid diagram, any arc on $\pi$ with endpoints on the rails and in general position with respect to them, with additional over/under data on its intersections with itself and the rails, except for its endpoints. 

We call two rail knotoid diagrams on $\pi$ as rail equivalent whenever one can be obtained from the other via a finite sequence of the usual Reidemester moves $\Omega_1, \ \Omega_2, \ \Omega_3$, along with the slide moves and planar isotopy moves or just planar isotopies of $\pi$, all defined locally as in Figure \ref{figure_rail_moves} with the provisions explained in the caption and where the moves involving the rails are given in separate subfigures for more clarity. We call these moves as rail knotoid moves.

Equivalence between rail knotoid diagrams as defined, is an equivalence relation in the set of all planar rail knotoid diagrams. We call the equivalence classes simply as planar rail knotoids or just rail knotoids. If we wish we can consider rail knotoids as oriented, with their orientation being the leg to head orientation in any one of their representatives.

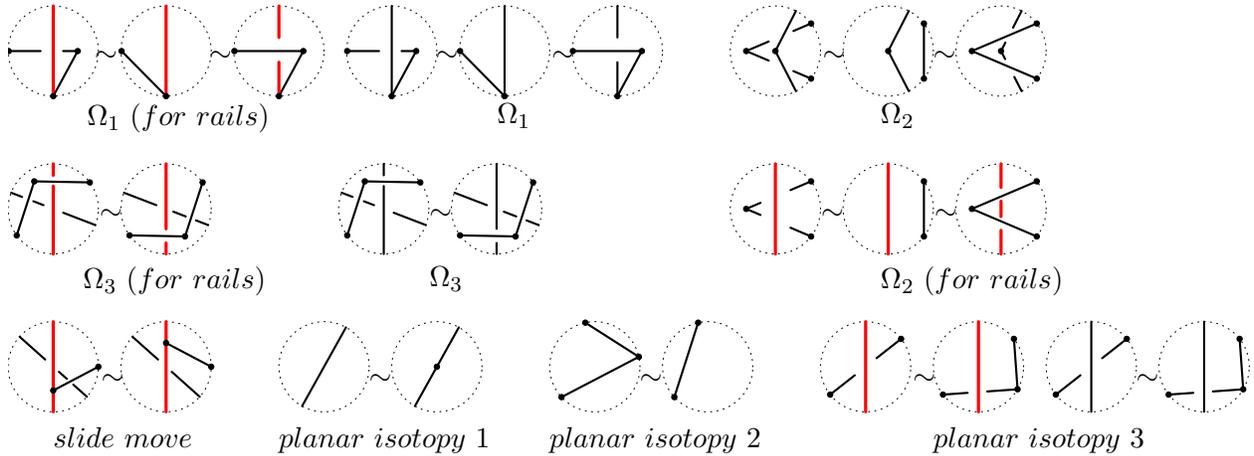
\begin{figure}[!h]
	\centering 
	\definecolor{ffqqqq}{rgb}{1.,0.,0.}
	\begin{tikzpicture}[line cap=round,line join=round,>=triangle 45,x=1.0cm,y=1.0cm]
	\clip(0.635,0.5040534152619682) rectangle (17.23932194149101,6.928883480699878);
	\draw [line width=0.4pt,dotted] (1.235839491929495,6.063072159460898) circle (0.6cm);
	\draw [line width=0.4pt,dotted] (1.235839491929495,3.9630721594608986) circle (0.6cm);
	\draw [line width=0.4pt,dotted] (1.235839491929495,1.8630721594608994) circle (0.6cm);
	\draw [line width=0.4pt,dotted] (2.7358394919294957,6.063072159460898) circle (0.6cm);
	\draw [line width=0.4pt,dotted] (4.235839491929497,6.063072159460898) circle (0.6cm);
	\draw [line width=0.4pt,dotted] (5.7358394919294975,6.063072159460898) circle (0.6cm);
	\draw [line width=0.4pt,dotted] (7.235839491929498,6.063072159460898) circle (0.6cm);
	\draw [line width=0.4pt,dotted] (8.7358394919295,6.063072159460898) circle (0.6cm);
	\draw [line width=0.4pt,dotted] (10.835839491929496,6.063072159460898) circle (0.6cm);
	\draw [line width=0.4pt,dotted] (12.335839491929496,6.063072159460898) circle (0.6cm);
	\draw [line width=0.4pt,dotted] (13.835839491929496,6.063072159460898) circle (0.6cm);
	\draw [line width=0.4pt,dotted] (2.7358394919294957,3.9630721594608986) circle (0.6cm);
	\draw [line width=0.4pt,dotted] (10.835839491929496,3.9630721594608986) circle (0.6cm);
	\draw [line width=0.4pt,dotted] (12.335839491929496,3.9630721594608986) circle (0.6cm);
	\draw [line width=0.4pt,dotted] (13.835839491929496,3.9630721594608986) circle (0.6cm);
	\draw [line width=0.4pt,dotted] (4.835839491929492,1.8630721594608994) circle (0.6cm);
	\draw [line width=0.4pt,dotted] (6.335839491929493,1.8630721594608994) circle (0.6cm);
	\draw [line width=0.4pt,dotted] (8.435839491929489,1.8630721594608994) circle (0.6cm);
	\draw [line width=0.4pt,dotted] (9.935839491929489,1.8630721594608994) circle (0.6cm);
	\draw [line width=0.4pt,dotted] (2.7358394919294957,1.8630721594608994) circle (0.6cm);
	\draw [line width=0.4pt,dotted] (12.035839491929485,1.8630721594608994) circle (0.6cm);
	\draw [line width=0.4pt,dotted] (13.535839491929485,1.8630721594608994) circle (0.6cm);
	\draw [line width=0.4pt,dotted] (15.035839491929485,1.8630721594608994) circle (0.6cm);
	\draw [line width=0.4pt,dotted] (16.535839491929487,1.8630721594608994) circle (0.6cm);
	\draw [line width=0.8pt] (0.6439933044901619,6.063072159460898)-- (1.067415882689213,6.063072159460898);
	\draw [line width=0.8pt] (1.366772156734188,6.063072159460898)-- (1.560931311568441,6.063072159460898);
	\draw [line width=0.8pt] (1.560931311568441,6.063072159460898)-- (1.235839491929495,5.463072159460898);
	\draw [line width=0.8pt] (2.1439933044901625,6.063072159460898)-- (2.7358394919294957,5.463072159460898);
	\draw [line width=0.8pt] (4.235839491929497,5.463072159460898)-- (4.560931311568442,6.063072159460898);
	\draw [line width=0.8pt] (4.560931311568442,6.063072159460898)-- (3.643993304490163,6.063072159460898);
	\draw [line width=0.8pt] (5.7358394919294975,6.6630721594608975)-- (5.7358394919294975,5.463072159460898);
	\draw [line width=0.8pt] (5.866772156734191,6.063072159460898)-- (6.060931311568443,6.063072159460898);
	\draw [line width=0.8pt] (6.060931311568443,6.063072159460898)-- (5.7358394919294975,5.463072159460898);
	\draw [line width=0.8pt] (5.143993304490164,6.063072159460898)-- (5.5674158826892155,6.063072159460898);
	\draw [line width=0.8pt] (6.643993304490165,6.063072159460898)-- (7.235839491929498,5.463072159460898);
	\draw [line width=0.8pt] (7.235839491929498,5.463072159460898)-- (7.235839491929498,6.6630721594608975);
	\draw [line width=0.8pt] (8.143993304490165,6.063072159460898)-- (9.060931311568446,6.063072159460898);
	\draw [line width=0.8pt] (9.060931311568446,6.063072159460898)-- (8.735839491929498,5.463072159460898);
	\draw [line width=1.2pt,color=ffqqqq] (4.235839491929497,5.463072159460898)-- (4.235839491929497,5.914513925353407);
	\draw [line width=1.2pt,color=ffqqqq] (4.235839491929497,6.243480228695809)-- (4.235839491929497,6.6630721594608975);
	\draw [line width=0.8pt] (8.735839491929498,5.463072159460898)-- (8.7358394919295,5.914513925353407);
	\draw [line width=0.8pt] (8.7358394919295,6.243480228695809)-- (8.735839491929498,6.6630721594608975);
	\draw [line width=1.2pt,color=ffqqqq] (1.235839491929495,5.463072159460898)-- (1.235839491929495,6.6630721594608975);
	\draw [line width=1.2pt,color=ffqqqq] (2.7358394919294957,6.6630721594608975)-- (2.7358394919294957,5.463072159460898);
	\draw [line width=0.8pt] (10.448350504394675,6.063072159460898)-- (10.74749215042923,6.19028015325822);
	\draw [line width=0.8pt] (11.07529611426485,6.3296766069152355)-- (11.310725728385822,6.429791483228201);
	\draw [line width=0.8pt] (10.74749215042923,5.935864165663578)-- (10.448350504394675,6.063072159460898);
	\draw [line width=0.8pt] (11.07529611426485,5.796467712006563)-- (11.310725728385822,5.696352835693597);
	\draw [line width=0.8pt] (11.109705105210379,5.529220470447661)-- (10.835839491929496,6.063072159460898);
	\draw [line width=0.8pt] (10.835839491929496,6.063072159460898)-- (11.109705105210379,6.5969238484741375);
	\draw [line width=0.8pt] (12.810725728385822,6.429791483228201)-- (12.810725728385822,5.696352835693597);
	\draw [line width=0.8pt] (12.335839491929496,6.063072159460898)-- (12.609705105210379,6.5969238484741375);
	\draw [line width=0.8pt] (12.609705105210379,5.529220470447661)-- (12.335839491929496,6.063072159460898);
	\draw [line width=0.8pt] (13.448350504394675,6.063072159460898)-- (14.310725728385822,6.429791483228201);
	\draw [line width=0.8pt] (14.310725728385822,5.696352835693597)-- (13.448350504394675,6.063072159460898);
	\draw [line width=0.8pt] (13.89400633627269,6.176457946472197)-- (13.835839491929496,6.063072159460898);
	\draw [line width=0.8pt] (13.835839491929496,6.063072159460898)-- (13.891007007193274,5.9555330243011095);
	\draw [line width=0.8pt] (13.995647196124736,6.374588496907969)-- (14.109705105210379,6.5969238484741375);
	\draw [line width=0.8pt] (14.012257251406721,5.719177510765978)-- (14.109705105210379,5.529220470447661);
	\draw [line width=0.8pt] (10.448350504394675,3.9630721594608986)-- (10.638966474141755,4.044130331797278);
	\draw [line width=0.8pt] (11.036220360031368,4.213059902527696)-- (11.310725728385822,4.3297914832282025);
	\draw [line width=0.8pt] (10.448350504394675,3.9630721594608986)-- (10.638966474141755,3.8820139871245196);
	\draw [line width=0.8pt] (11.036220360031368,3.7130844163941017)-- (11.310725728385822,3.596352835693598);
	\draw [line width=0.8pt] (12.810725728385822,4.3297914832282025)-- (12.810725728385822,3.596352835693598);
	\draw [line width=0.8pt] (13.448350504394675,3.9630721594608986)-- (14.310725728385822,4.3297914832282025);
	\draw [line width=0.8pt] (13.448350504394675,3.9630721594608986)-- (14.310725728385822,3.596352835693598);
	\draw [line width=1.2pt,color=ffqqqq] (10.835839491929496,4.563072159460898)-- (10.835839491929496,3.3630721594608994);
	\draw [line width=1.2pt,color=ffqqqq] (12.335839491929496,4.563072159460898)-- (12.335839491929496,3.3630721594608994);
	\draw [line width=1.2pt,color=ffqqqq] (13.835839491929496,4.563072159460898)-- (13.835839491929496,4.2481223117461155);
	\draw [line width=1.2pt,color=ffqqqq] (13.835839491929496,4.065555936199785)-- (13.835839491929496,3.860588382722012);
	\draw [line width=1.2pt,color=ffqqqq] (13.835839491929496,3.3630721594608994)-- (13.835839491929496,3.6780220071756817);
	\draw [line width=0.8pt] (0.983640087392172,4.3293591330133285)-- (0.7502996002155788,3.610579587121511);
	\draw [line width=0.8pt] (0.983640087392172,4.3293591330133285)-- (1.7213793836434113,4.315564731800286);
	\draw [line width=0.8pt] (0.6757984339441674,4.1783719815014075)-- (0.8219479169670243,4.122186890716585);
	\draw [line width=0.8pt] (0.9709944835317432,4.0648880573550255)-- (1.1068854795988603,4.012646696024471);
	\draw [line width=0.8pt] (1.367712812830891,3.9123753362242257)-- (1.7958805499148227,3.7477723374203897);
	\draw [line width=1.2pt,color=ffqqqq] (1.235839491929495,3.3630721594608977)-- (1.235839491929495,4.2670216486901325);
	\draw [line width=1.2pt,color=ffqqqq] (1.235839491929495,4.409624190671764)-- (1.235839491929495,4.5630721594609);
	\draw [line width=0.8pt] (2.9880388964668185,3.5967851859084687)-- (2.2502996002155795,3.610579587121511);
	\draw [line width=0.8pt] (2.9880388964668185,3.5967851859084687)-- (3.221379383643412,4.315564731800286);
	\draw [line width=0.8pt] (3.1320689064953533,3.810747394647889)-- (3.2958805499148234,3.7477723374203897);
	\draw [line width=1.2pt,color=ffqqqq] (2.7358394919294957,4.5630721594609)-- (2.7358394919294957,3.6971285763375916);
	\draw [line width=1.2pt,color=ffqqqq] (2.7358394919294957,3.5100790864908467)-- (2.7358394919294957,3.3630721594608977);
	\draw [line width=0.8pt] (2.606885479598861,4.012646696024471)-- (2.175798433944168,4.1783719815014075);
	\draw [line width=0.8pt] (2.83197520110087,3.926114152934032)-- (2.997368963082957,3.8625308727198044);
	\draw [line width=0.8pt] (1.235839491929495,1.5469862693104748)-- (1.8358394919294954,1.8630721594608994);
	\draw [line width=0.8pt] (3.3358394919294967,1.8630721594608994)-- (2.7358394919294957,2.179158049611324);
	\draw [line width=0.8pt] (0.7882672906370293,2.262670863951062)-- (1.128909514193796,1.9585407351475874);
	\draw [line width=0.8pt] (1.416366234123973,1.7018953658655254)-- (1.3108212870879676,1.7961273668237594);
	\draw [line width=0.8pt] (1.5020987260540666,1.6253522011894144)-- (1.6834116932219607,1.4634734549707367);
	\draw [line width=0.8pt] (2.28826729063703,2.262670863951062)-- (2.6289095141937966,1.9585407351475874);
	\draw [line width=0.8pt] (2.8108212870879683,1.7961273668237594)-- (3.183411693221961,1.4634734549707367);
	\draw [line width=1.2pt,color=ffqqqq] (1.235839491929495,2.4630721594609)-- (1.235839491929495,1.2630721594608991);
	\draw [line width=1.2pt,color=ffqqqq] (2.7358394919294957,2.4630721594609)-- (2.7358394919294957,1.2630721594608991);
	\draw [line width=0.8pt] (6.627852526972036,2.3872175470827335)-- (6.04382645688695,1.3389267718390654);
	\draw [line width=0.8pt] (5.127852526972035,2.3872175470827335)-- (4.543826456886949,1.3389267718390654);
	\draw [line width=0.8pt] (7.992515329705296,1.458765732121168)-- (9.021189864965,1.994848260432155);
	\draw [line width=0.8pt] (9.021189864965,1.994848260432155)-- (8.309346935799486,2.449586978748271);
	\draw [line width=0.8pt] (9.809346935799486,2.449586978748271)-- (9.492515329705295,1.458765732121168);
	\draw [line width=1.2pt,color=ffqqqq] (12.035839491929485,2.463072159460899)-- (12.035839491929485,1.2630721594608998);
	\draw [line width=0.8pt] (11.564153843997596,1.4922451788687718)-- (11.891901608530231,1.7499119415746973);
	\draw [line width=0.8pt] (12.179777375328738,1.9762323773471016)-- (12.507525139861373,2.233899140053027);
	\draw [line width=0.8pt] (14.055519126223329,1.5631837173691507)-- (14.007525139861373,2.233899140053027);
	\draw [line width=1.2pt,color=ffqqqq] (13.535839491929485,2.463072159460899)-- (13.535839491929485,1.2630721594608998);
	\draw [line width=0.8pt] (13.064153843997596,1.4922451788687718)-- (13.392372759042894,1.515731345430122);
	\draw [line width=0.8pt] (13.651717468756951,1.5342891212618675)-- (14.055519126223329,1.5631837173691507);
	\draw [line width=0.8pt] (15.035839491929485,2.463072159460899)-- (15.035839491929485,1.2630721594608998);
	\draw [line width=0.8pt] (14.564153843997596,1.4922451788687718)-- (14.891901608530231,1.7499119415746973);
	\draw [line width=0.8pt] (15.179777375328738,1.9762323773471016)-- (15.507525139861373,2.233899140053027);
	\draw [line width=0.8pt] (16.064153843997598,1.4922451788687718)-- (16.392372759042896,1.515731345430122);
	\draw [line width=0.8pt] (16.535839491929487,2.463072159460899)-- (16.535839491929487,1.2630721594608998);
	\draw [line width=0.8pt] (16.651717468756953,1.5342891212618675)-- (17.05551912622333,1.5631837173691507);
	\draw [line width=0.8pt] (17.05551912622333,1.5631837173691507)-- (17.007525139861375,2.233899140053027);
	\draw [line width=0.4pt,dotted] (5.629631975570058,3.9630721594608986) circle (0.6cm);
	\draw [line width=0.8pt] (5.144092083856142,3.610579587121511)-- (5.377432571032736,4.3293591330133285);
	\draw [line width=0.8pt] (5.377432571032736,4.3293591330133285)-- (6.1151718672839745,4.315564731800286);
	\draw [line width=0.8pt] (5.629631975570058,3.3630721594608977)-- (5.629631975570058,4.2670216486901325);
	\draw [line width=0.8pt] (5.629631975570058,4.409624190671764)-- (5.629631975570058,4.5630721594609);
	\draw [line width=0.8pt] (5.069590917584732,4.1783719815014075)-- (5.215740400607588,4.122186890716585);
	\draw [line width=0.8pt] (5.364786967172307,4.0648880573550255)-- (5.500677963239424,4.012646696024471);
	\draw [line width=0.8pt] (5.761505296471455,3.9123753362242257)-- (6.189673033555387,3.7477723374203897);
	\draw [line width=0.4pt,dotted] (7.12963197557006,3.9630721594608986) circle (0.6cm);
	\draw [line width=0.8pt] (6.644092083856144,3.610579587121511)-- (7.381831380107382,3.5967851859084687);
	\draw [line width=0.8pt] (7.381831380107382,3.5967851859084687)-- (7.615171867283976,4.315564731800286);
	\draw [line width=0.8pt] (6.569590917584733,4.1783719815014075)-- (7.000677963239426,4.012646696024471);
	\draw [line width=0.8pt] (7.12963197557006,4.5630721594609)-- (7.12963197557006,3.6971285763375916);
	\draw [line width=0.8pt] (7.12963197557006,3.3630721594608977)-- (7.12963197557006,3.5100790864908467);
	\draw [line width=0.8pt] (7.2257676847414345,3.926114152934032)-- (7.3911614467235225,3.8625308727198044);
	\draw [line width=0.8pt] (7.525861390135917,3.810747394647889)-- (7.689673033555387,3.7477723374203897);

	\draw (1.652,6.2) node[anchor=north west] {$\sim$};
	\draw (3.173,6.2) node[anchor=north west] {$\sim$};
	\draw (6.185,6.2) node[anchor=north west] {$\sim$};
	\draw (7.73,6.2) node[anchor=north west] {$\sim$};
	\draw (11.33,6.2) node[anchor=north west] {$\sim$};
	\draw (12.8,6.2) node[anchor=north west] {$\sim$};
	\draw (1.72,4.1) node[anchor=north west] {$\sim$};
	\draw (6.1,4.1) node[anchor=north west] {$\sim$};
	\draw (11.31,4.1) node[anchor=north west] {$\sim$};
	\draw (12.82,4.1) node[anchor=north west] {$\sim$};
	\draw (1.742,1.9) node[anchor=north west] {$\sim$};
	\draw (5.3,1.9) node[anchor=north west] {$\sim$};
	\draw (8.9,1.9) node[anchor=north west] {$\sim$};
	\draw (12.53,1.9) node[anchor=north west] {$\sim$};
	\draw (15.5,1.9) node[anchor=north west] {$\sim$};
	\draw (7,5.5) node[anchor=north west] {$\Omega_1$};
	\draw (1.55,5.5) node[anchor=north west] {$\Omega_1\ (for \ rails)$};
	\draw (12.1,5.5) node[anchor=north west] {$\Omega_2$};
	\draw (6.1,3.33) node[anchor=north west] {$\Omega_3$};
	\draw (12.1,3.33) node[anchor=north west] {$\Omega_2 \ (for \ rails)$};
	\draw (1.1,1.2) node[anchor=north west] {$slide \ move$};
	\draw (4.1,1.2) node[anchor=north west] {$planar \ isotopy \ 1$};
	\draw (7.7,1.2) node[anchor=north west] {$planar \ isotopy \ 2$};
	\draw (12.8,1.2) node[anchor=north west] {$planar \ isotopy \ 3$};
	\draw (1.5,3.33) node[anchor=north west] {$\Omega_3\ (for \ rails)$};

	\begin{scriptsize}
	\draw [fill=black] (10.835839491929496,6.063072159460898) circle (1.0pt);
	\draw [fill=black] (12.335839491929496,6.063072159460898) circle (1.0pt);
	\draw [fill=black] (13.835839491929496,6.063072159460898) circle (1.0pt);
	\draw [fill=black] (-11.364160508070507,3.9630721594608986) circle (1.0pt);
	\draw [fill=black] (6.335839491929493,1.8630721594608994) circle (1.0pt);
	\draw [fill=black] (1.235839491929495,5.463072159460898) circle (1.0pt);
	\draw [fill=black] (0.6439933044901619,6.063072159460898) circle (1.0pt);
	\draw [fill=black] (1.560931311568441,6.063072159460898) circle (1.0pt);
	\draw [fill=black] (2.1439933044901625,6.063072159460898) circle (1.0pt);
	\draw [fill=black] (2.7358394919294957,5.463072159460898) circle (1.0pt);
	\draw [fill=black] (3.643993304490163,6.063072159460898) circle (1.0pt);
	\draw [fill=black] (4.560931311568442,6.063072159460898) circle (1.0pt);
	\draw [fill=black] (5.143993304490164,6.063072159460898) circle (1.0pt);
	\draw [fill=black] (6.060931311568443,6.063072159460898) circle (1.0pt);
	\draw [fill=black] (4.235839491929497,5.463072159460898) circle (1.0pt);
	\draw [fill=black] (5.7358394919294975,5.463072159460898) circle (1.0pt);
	\draw [fill=black] (7.235839491929498,5.463072159460898) circle (1.0pt);
	\draw [fill=black] (8.735839491929498,5.463072159460898) circle (1.0pt);
	\draw [fill=black] (6.643993304490165,6.063072159460898) circle (1.0pt);
	\draw [fill=black] (8.143993304490165,6.063072159460898) circle (1.0pt);
	\draw [fill=black] (9.060931311568446,6.063072159460898) circle (1.0pt);
	\draw [fill=black] (10.448350504394675,6.063072159460898) circle (1.0pt);
	\draw [fill=black] (11.310725728385822,5.696352835693597) circle (1.0pt);
	\draw [fill=black] (11.310725728385822,6.429791483228201) circle (1.0pt);
	\draw [fill=black] (13.448350504394675,6.063072159460898) circle (1.0pt);
	\draw [fill=black] (12.810725728385822,6.429791483228201) circle (1.0pt);
	\draw [fill=black] (12.810725728385822,5.696352835693597) circle (1.0pt);
	\draw [fill=black] (14.310725728385822,6.429791483228201) circle (1.0pt);
	\draw [fill=black] (14.310725728385822,5.696352835693597) circle (1.0pt);
	\draw [fill=black] (10.448350504394675,3.9630721594608986) circle (1.0pt);
	\draw [fill=black] (11.310725728385822,4.3297914832282025) circle (1.0pt);
	\draw [fill=black] (11.310725728385822,3.596352835693598) circle (1.0pt);
	\draw [fill=black] (12.810725728385822,4.3297914832282025) circle (1.0pt);
	\draw [fill=black] (12.810725728385822,3.596352835693598) circle (1.0pt);
	\draw [fill=black] (14.310725728385822,4.3297914832282025) circle (1.0pt);
	\draw [fill=black] (14.310725728385822,3.596352835693598) circle (1.0pt);
	\draw [fill=black] (13.448350504394675,3.9630721594608986) circle (1.0pt);
	\draw [fill=black] (0.983640087392172,4.3293591330133285) circle (1.0pt);
	\draw [fill=black] (0.7502996002155788,3.610579587121511) circle (1.0pt);
	\draw [fill=black] (1.7213793836434113,4.315564731800286) circle (1.0pt);
	\draw [fill=black] (2.2502996002155795,3.610579587121511) circle (1.0pt);
	\draw [fill=black] (3.221379383643412,4.315564731800286) circle (1.0pt);
	\draw [fill=black] (2.9880388964668185,3.5967851859084687) circle (1.0pt);
	\draw [fill=black] (3.3358394919294967,1.8630721594608994) circle (1.0pt);
	\draw [fill=black] (1.235839491929495,1.5469862693104748) circle (1.0pt);
	\draw [fill=black] (1.8358394919294954,1.8630721594608994) circle (1.0pt);
	\draw [fill=black] (2.7358394919294957,2.179158049611324) circle (1.0pt);
	\draw [fill=black] (7.992515329705296,1.458765732121168) circle (1.0pt);
	\draw [fill=black] (8.309346935799486,2.449586978748271) circle (1.0pt);
	\draw [fill=black] (9.021189864965,1.994848260432155) circle (1.0pt);
	\draw [fill=black] (9.492515329705295,1.458765732121168) circle (1.0pt);
	\draw [fill=black] (9.809346935799486,2.449586978748271) circle (1.0pt);
	\draw [fill=black] (11.564153843997596,1.4922451788687718) circle (1.0pt);
	\draw [fill=black] (12.507525139861373,2.233899140053027) circle (1.0pt);
	\draw [fill=black] (13.064153843997596,1.4922451788687718) circle (1.0pt);
	\draw [fill=black] (14.055519126223329,1.5631837173691507) circle (1.0pt);
	\draw [fill=black] (14.007525139861373,2.233899140053027) circle (1.0pt);
	\draw [fill=black] (14.564153843997596,1.4922451788687718) circle (1.0pt);
	\draw [fill=black] (15.507525139861373,2.233899140053027) circle (1.0pt);
	\draw [fill=black] (16.064153843997598,1.4922451788687718) circle (1.0pt);
	\draw [fill=black] (17.05551912622333,1.5631837173691507) circle (1.0pt);
	\draw [fill=black] (17.007525139861375,2.233899140053027) circle (1.0pt);
	\draw [fill=black] (5.377432571032736,4.3293591330133285) circle (1.0pt);
	\draw [fill=black] (6.1151718672839745,4.315564731800286) circle (1.0pt);
	\draw [fill=black] (6.644092083856144,3.610579587121511) circle (1.0pt);
	\draw [fill=black] (7.381831380107382,3.5967851859084687) circle (1.0pt);
	\draw [fill=black] (7.615171867283976,4.315564731800286) circle (1.0pt);
	\draw [fill=black] (5.144092083856142,3.610579587121511) circle (1.0pt);
	\end{scriptsize}
	\end{tikzpicture}

	\caption{Red lines denote rails and dots denote vertices. For clarity, the moves involving rails were given in separate figures. The moving part in $\Omega_3$ moves can be on top, on bottom or in the middle, and similarly for the rail. The moving part and the rail in slide moves can be on top or on bottom, and the moving part can be on the left or the right of the rail. The moving part in planar isotopy 3 moves can be on top or bottom and the disappearing vertex on the left or right of the fixed part.}
\label{figure_rail_moves}
\end{figure}

 In \cite{KL} we proved that our aforementioned expectation about rail isotopy indeed holds:

\begin{theorem*} 
	Two rail arcs in $\mathbb{R}^3$ are rail isotopic iff their rail knotoid diagram projections on the plane $\pi$ of the rails are rail equivalent. In other words, rail isotopy in $\mathbb{R}^3$ corresponds to rail equivalence on $\pi$ (rail arcs are isotopic iff they correspond to the same rail knotoid).
\end{theorem*}

\section{Defining knots corresponding to rail knotoids}
Let $\kappa$ be a rail knotoid. We are going to define and correspond to $\kappa$ two  unoriented knots as well as their oriented versions. We start with some representative rail  knotoid diagram $K$ for $\kappa$ and we first define $K$'s over and under companion loops as follows:

\begin{definition} \rm Let $K$ be a rail knotoid diagram.  We call as  \textit{over companion loop} $K_o$ of $K$, respectively \textit{under companion loop} $K_u$ of $K$, any loop in space whose knot projection on the plane  $\pi$ of the rails is described as follows (see Figure \ref{figure_companion_loops}):

\begin{itemize}
	\item Knot projection of $K_o$: (1) a line segment with its endpoints on the rails and vertical to them, put high up enough so that it does not mess with $K$, (2) one segment on each rail joining the corresponding endpoint of the vertical segment with the leg or the head respectively, (3) the rail knotoid diagram $K$ with the over/under data of $K$ at the crossing points with the segments projecting on the rails. 
	
	\item Knot projection of $K_u$: as for $K_o$, only now we place the vertical segment in (1) low down enough so that it does not mess with $K$.
			
	\end{itemize}

We denote the oriented versions of the chosen  loops $K_o,K_u$ as $K_{o+},K_{o-},K_{u+},K_{u-}$ where the $+$ sign indicates the orientation induced by the rail knotoid and the $-$ sign indicates the opposite. We call these as \textit{oriented} over or under companion loops of $K$ respectively.

\end{definition}

\begin{figure}[!h]
	\centering
	\definecolor{ffqqqq}{rgb}{1.,0.,0.}
	\begin{tikzpicture}[line cap=round,line join=round,>=triangle 45,x=1.0cm,y=1.0cm]
	\clip(3.6398049769069836,4.276636051040948) rectangle (16.2,8.200067153741754);
\draw [shift={(6.165968007865688,6.320138130917781)},line width=0.4pt]  plot[domain=-0.9090662575001538:0.9923514156973319,variable=\t]({1.*0.24226291935302283*cos(\t r)+0.*0.24226291935302283*sin(\t r)},{0.*0.24226291935302283*cos(\t r)+1.*0.24226291935302283*sin(\t r)});
\draw [shift={(5.219811882128152,4.820081265772829)},line width=0.4pt]  plot[domain=1.2981046296861078:1.718995016247648,variable=\t]({1.*2.023928926136264*cos(\t r)+0.*2.023928926136264*sin(\t r)},{0.*2.023928926136264*cos(\t r)+1.*2.023928926136264*sin(\t r)});
\draw [shift={(6.204868053391834,6.34658325187964)},line width=0.4pt]  plot[domain=1.8031194837987232:3.0144356198078617,variable=\t]({1.*0.48491668043597475*cos(\t r)+0.*0.48491668043597475*sin(\t r)},{0.*0.48491668043597475*cos(\t r)+1.*0.48491668043597475*sin(\t r)});
\draw [shift={(6.309704444107993,6.62558684206024)},line width=0.4pt]  plot[domain=-0.8893986781627756:0.8893986781627703,variable=\t]({1.*0.21663861102863105*cos(\t r)+0.*0.21663861102863105*sin(\t r)},{0.*0.21663861102863105*cos(\t r)+1.*0.21663861102863105*sin(\t r)});
\draw [shift={(6.02638382590094,6.382174114055326)},line width=0.4pt]  plot[domain=3.0561739652779294:4.962142713419938,variable=\t]({1.*0.3036244470810311*cos(\t r)+0.*0.3036244470810311*sin(\t r)},{0.*0.3036244470810311*cos(\t r)+1.*0.3036244470810311*sin(\t r)});
\draw [shift={(5.913020909225199,6.165646600533312)},line width=0.4pt]  plot[domain=0.7476423546557244:1.4462476259445392,variable=\t]({1.*0.5255705655408514*cos(\t r)+0.*0.5255705655408514*sin(\t r)},{0.*0.5255705655408514*cos(\t r)+1.*0.5255705655408514*sin(\t r)});
\draw [shift={(6.234937330459185,6.3080246633316115)},line width=0.4pt]  plot[domain=1.1606768062067996:1.8416066999258576,variable=\t]({1.*0.5297547816122091*cos(\t r)+0.*0.5297547816122091*sin(\t r)},{0.*0.5297547816122091*cos(\t r)+1.*0.5297547816122091*sin(\t r)});
\draw [shift={(4.575785902501749,5.879313862501207)},line width=0.4pt]  plot[domain=5.043761529785609:6.22781696755429,variable=\t]({1.*0.5549087482043962*cos(\t r)+0.*0.5549087482043962*sin(\t r)},{0.*0.5549087482043962*cos(\t r)+1.*0.5549087482043962*sin(\t r)});
\draw [shift={(4.916366821999889,5.796893604784073)},line width=0.4pt]  plot[domain=0.23765657725350492:1.9992233526240046,variable=\t]({1.*0.2196513515551546*cos(\t r)+0.*0.2196513515551546*sin(\t r)},{0.*0.2196513515551546*cos(\t r)+1.*0.2196513515551546*sin(\t r)});
\draw [shift={(4.87084544220089,5.756343328057861)},line width=0.4pt]  plot[domain=2.6604632973133926:5.083867319629343,variable=\t]({1.*0.2120285880272482*cos(\t r)+0.*0.2120285880272482*sin(\t r)},{0.*0.2120285880272482*cos(\t r)+1.*0.2120285880272482*sin(\t r)});
\draw [shift={(5.337200378913295,6.745049224444948)},line width=0.4pt]  plot[domain=4.9134570762437075:5.832726076495774,variable=\t]({1.*1.1133137627604077*cos(\t r)+0.*1.1133137627604077*sin(\t r)},{0.*1.1133137627604077*cos(\t r)+1.*1.1133137627604077*sin(\t r)});
\draw [shift={(4.642083259921757,9.789377117550492)},line width=0.4pt]  plot[domain=4.827682869882762:4.930718586642295,variable=\t]({1.*4.235767135465101*cos(\t r)+0.*4.235767135465101*sin(\t r)},{0.*4.235767135465101*cos(\t r)+1.*4.235767135465101*sin(\t r)});
\draw [line width=1.2pt,color=ffqqqq] (4.756320587529386,5.521292638718075)-- (4.756320587529386,4.605041317096833);
\draw [shift={(5.146852170408313,5.4731930892554175)},line width=0.4pt]  plot[domain=1.9216994299710348:2.790394803642949,variable=\t]({1.*1.36311380979352*cos(\t r)+0.*1.36311380979352*sin(\t r)},{0.*1.36311380979352*cos(\t r)+1.*1.36311380979352*sin(\t r)});
\draw [line width=1.2pt,color=ffqqqq] (4.756320587529386,5.700383831150284)-- (4.756320587529386,8.10919846858383);
\draw [line width=1.2pt,color=ffqqqq] (3.866941057018287,8.10919846858383)-- (3.866941057018287,4.605041317096833);
\draw [shift={(10.933984115515889,6.320138130917781)},line width=0.4pt]  plot[domain=-0.9090662575001653:0.9923514156973472,variable=\t]({1.*0.24226291935302063*cos(\t r)+0.*0.24226291935302063*sin(\t r)},{0.*0.24226291935302063*cos(\t r)+1.*0.24226291935302063*sin(\t r)});
\draw [shift={(9.98782798977835,4.820081265772817)},line width=0.4pt]  plot[domain=1.2981046296861094:1.7189950162476466,variable=\t]({1.*2.023928926136275*cos(\t r)+0.*2.023928926136275*sin(\t r)},{0.*2.023928926136275*cos(\t r)+1.*2.023928926136275*sin(\t r)});
\draw [shift={(10.97288416104203,6.346583251879643)},line width=0.4pt]  plot[domain=1.8031194837987214:3.0144356198078683,variable=\t]({1.*0.48491668043597086*cos(\t r)+0.*0.48491668043597086*sin(\t r)},{0.*0.48491668043597086*cos(\t r)+1.*0.48491668043597086*sin(\t r)});
\draw [shift={(11.077720551758189,6.62558684206024)},line width=0.4pt]  plot[domain=-0.8893986781627659:0.8893986781627607,variable=\t]({1.*0.21663861102863272*cos(\t r)+0.*0.21663861102863272*sin(\t r)},{0.*0.21663861102863272*cos(\t r)+1.*0.21663861102863272*sin(\t r)});
\draw [shift={(10.794399933551137,6.382174114055323)},line width=0.4pt]  plot[domain=3.0561739652779205:4.96214271341994,variable=\t]({1.*0.3036244470810285*cos(\t r)+0.*0.3036244470810285*sin(\t r)},{0.*0.3036244470810285*cos(\t r)+1.*0.3036244470810285*sin(\t r)});
\draw [shift={(10.681037016875397,6.165646600533317)},line width=0.4pt]  plot[domain=0.7476423546557193:1.4462476259445411,variable=\t]({1.*0.5255705655408464*cos(\t r)+0.*0.5255705655408464*sin(\t r)},{0.*0.5255705655408464*cos(\t r)+1.*0.5255705655408464*sin(\t r)});
\draw [shift={(11.002953438109383,6.3080246633316115)},line width=0.4pt]  plot[domain=1.1606768062067996:1.8416066999258591,variable=\t]({1.*0.5297547816122093*cos(\t r)+0.*0.5297547816122093*sin(\t r)},{0.*0.5297547816122093*cos(\t r)+1.*0.5297547816122093*sin(\t r)});
\draw [shift={(9.34380201015194,5.879313862501209)},line width=0.4pt]  plot[domain=5.043761529785618:6.227816967554287,variable=\t]({1.*0.5549087482043998*cos(\t r)+0.*0.5549087482043998*sin(\t r)},{0.*0.5549087482043998*cos(\t r)+1.*0.5549087482043998*sin(\t r)});
\draw [shift={(9.684382929650086,5.796893604784073)},line width=0.4pt]  plot[domain=0.23765657725350492:1.999223352624001,variable=\t]({1.*0.2196513515551546*cos(\t r)+0.*0.2196513515551546*sin(\t r)},{0.*0.2196513515551546*cos(\t r)+1.*0.2196513515551546*sin(\t r)});
\draw [shift={(9.638861549851088,5.756343328057859)},line width=0.4pt]  plot[domain=2.660463297313385:5.083867319629341,variable=\t]({1.*0.21202858802724903*cos(\t r)+0.*0.21202858802724903*sin(\t r)},{0.*0.21202858802724903*cos(\t r)+1.*0.21202858802724903*sin(\t r)});
\draw [shift={(10.105216486563487,6.745049224444951)},line width=0.4pt]  plot[domain=4.913457076243713:5.832726076495773,variable=\t]({1.*1.1133137627604122*cos(\t r)+0.*1.1133137627604122*sin(\t r)},{0.*1.1133137627604122*cos(\t r)+1.*1.1133137627604122*sin(\t r)});
\draw [shift={(9.410099367571933,9.789377117550597)},line width=0.4pt]  plot[domain=4.827682869882764:4.930718586642295,variable=\t]({1.*4.235767135465208*cos(\t r)+0.*4.235767135465208*sin(\t r)},{0.*4.235767135465208*cos(\t r)+1.*4.235767135465208*sin(\t r)});
\draw [shift={(9.914868278058512,5.4731930892554175)},line width=0.4pt]  plot[domain=1.9216994299710355:2.790394803642949,variable=\t]({1.*1.3631138097935205*cos(\t r)+0.*1.3631138097935205*sin(\t r)},{0.*1.3631138097935205*cos(\t r)+1.*1.3631138097935205*sin(\t r)});
\draw [shift={(15.702000223166088,6.320138130917782)},line width=0.4pt]  plot[domain=-0.909066257500176:0.9923514156973493,variable=\t]({1.*0.24226291935302097*cos(\t r)+0.*0.24226291935302097*sin(\t r)},{0.*0.24226291935302097*cos(\t r)+1.*0.24226291935302097*sin(\t r)});
\draw [shift={(14.755844097428549,4.820081265772835)},line width=0.4pt]  plot[domain=1.2981046296861078:1.7189950162476488,variable=\t]({1.*2.023928926136257*cos(\t r)+0.*2.023928926136257*sin(\t r)},{0.*2.023928926136257*cos(\t r)+1.*2.023928926136257*sin(\t r)});
\draw [shift={(15.740900268692231,6.346583251879641)},line width=0.4pt]  plot[domain=1.8031194837987272:3.014435619807864,variable=\t]({1.*0.48491668043597436*cos(\t r)+0.*0.48491668043597436*sin(\t r)},{0.*0.48491668043597436*cos(\t r)+1.*0.48491668043597436*sin(\t r)});
\draw [shift={(15.845736659408384,6.62558684206024)},line width=0.4pt]  plot[domain=-0.8893986781627596:0.8893986781627543,variable=\t]({1.*0.21663861102863388*cos(\t r)+0.*0.21663861102863388*sin(\t r)},{0.*0.21663861102863388*cos(\t r)+1.*0.21663861102863388*sin(\t r)});
\draw [shift={(15.562416041201335,6.382174114055323)},line width=0.4pt]  plot[domain=3.0561739652779205:4.96214271341994,variable=\t]({1.*0.3036244470810285*cos(\t r)+0.*0.3036244470810285*sin(\t r)},{0.*0.3036244470810285*cos(\t r)+1.*0.3036244470810285*sin(\t r)});
\draw [shift={(15.449053124525589,6.165646600533309)},line width=0.4pt]  plot[domain=0.7476423546557235:1.4462476259445332,variable=\t]({1.*0.5255705655408558*cos(\t r)+0.*0.5255705655408558*sin(\t r)},{0.*0.5255705655408558*cos(\t r)+1.*0.5255705655408558*sin(\t r)});
\draw [shift={(15.770969545759582,6.308024663331606)},line width=0.4pt]  plot[domain=1.1606768062068067:1.8416066999258596,variable=\t]({1.*0.529754781612215*cos(\t r)+0.*0.529754781612215*sin(\t r)},{0.*0.529754781612215*cos(\t r)+1.*0.529754781612215*sin(\t r)});
\draw [shift={(14.111818117802141,5.879313862501209)},line width=0.4pt]  plot[domain=5.043761529785613:6.227816967554287,variable=\t]({1.*0.5549087482043987*cos(\t r)+0.*0.5549087482043987*sin(\t r)},{0.*0.5549087482043987*cos(\t r)+1.*0.5549087482043987*sin(\t r)});
\draw [shift={(14.452399037300284,5.796893604784074)},line width=0.4pt]  plot[domain=0.23765657725350098:1.9992233526240026,variable=\t]({1.*0.21965135155515436*cos(\t r)+0.*0.21965135155515436*sin(\t r)},{0.*0.21965135155515436*cos(\t r)+1.*0.21965135155515436*sin(\t r)});
\draw [shift={(14.406877657501285,5.756343328057859)},line width=0.4pt]  plot[domain=2.660463297313385:5.083867319629341,variable=\t]({1.*0.21202858802724903*cos(\t r)+0.*0.21202858802724903*sin(\t r)},{0.*0.21202858802724903*cos(\t r)+1.*0.21202858802724903*sin(\t r)});
\draw [shift={(14.87323259421369,6.74504922444495)},line width=0.4pt]  plot[domain=4.913457076243708:5.8327260764957725,variable=\t]({1.*1.1133137627604095*cos(\t r)+0.*1.1133137627604095*sin(\t r)},{0.*1.1133137627604095*cos(\t r)+1.*1.1133137627604095*sin(\t r)});
\draw [shift={(14.178115475222127,9.789377117550597)},line width=0.4pt]  plot[domain=4.827682869882764:4.930718586642296,variable=\t]({1.*4.235767135465208*cos(\t r)+0.*4.235767135465208*sin(\t r)},{0.*4.235767135465208*cos(\t r)+1.*4.235767135465208*sin(\t r)});
\draw [shift={(14.682884385708704,5.47319308925542)},line width=0.4pt]  plot[domain=1.9216994299710326:2.7903948036429496,variable=\t]({1.*1.3631138097935147*cos(\t r)+0.*1.3631138097935147*sin(\t r)},{0.*1.3631138097935147*cos(\t r)+1.*1.3631138097935147*sin(\t r)});
\draw [line width=0.4pt] (8.634957164668485,5.942135303467903)-- (8.634957164668485,7.543009720604965);
\draw [line width=0.4pt] (8.634957164668485,7.543009720604965)-- (9.524336695179583,7.543009720604965);
\draw [line width=0.4pt] (9.524336695179583,7.543009720604965)-- (9.524336695179583,5.700383831150284);
\draw [line width=0.4pt] (13.402973272318683,5.942135303467903)-- (13.402973272318683,4.95308318971556);
\draw [line width=0.4pt] (13.402973272318683,4.95308318971556)-- (14.29235280282978,4.95308318971556);
\draw [line width=0.4pt] (9.524336695179583,5.354593988343822)-- (9.524336695179583,5.521292638718075);
\draw [line width=0.4pt] (14.29235280282978,4.95308318971556)-- (14.29235280282978,5.354593988343822);
\draw [shift={(14.419229980704673,5.796383035836265)},line width=0.4pt]  plot[domain=1.853022153193996:2.859366827190717,variable=\t]({1.*0.20856101225148985*cos(\t r)+0.*0.20856101225148985*sin(\t r)},{0.*0.20856101225148985*cos(\t r)+1.*0.20856101225148985*sin(\t r)});
\draw [shift={(14.561745163801394,5.987534217235471)},line width=0.4pt]  plot[domain=1.695696026710191:1.9967488342716746,variable=\t]({1.*0.8408410087315366*cos(\t r)+0.*0.8408410087315366*sin(\t r)},{0.*0.8408410087315366*cos(\t r)+1.*0.8408410087315366*sin(\t r)});
\draw (3.8,4.947626643426834) node[anchor=north west] {$\ell_1$};
\draw (4.795399886015987,4.901030074511148) node[anchor=north west] {$\ell_2$};
\draw (5.280004202739118,7.21) node[anchor=north west] {$K$};
\draw (9.6,7.305413030560573) node[anchor=north west] {$K_o$};
\draw (14.6,7.305413030560573) node[anchor=north west] {$K_u$};
\begin{scriptsize}
\draw [fill=black] (4.756320587529386,5.354593988343822) circle (1.0pt);
\draw [fill=black] (3.866941057018287,5.942135303467903) circle (1.0pt);
\end{scriptsize}
\end{tikzpicture}
\caption{A rail knotoid diagram $K$ and a choice of over and under companion loops $K_o,K_u$.}
\label{figure_companion_loops}
\end{figure}
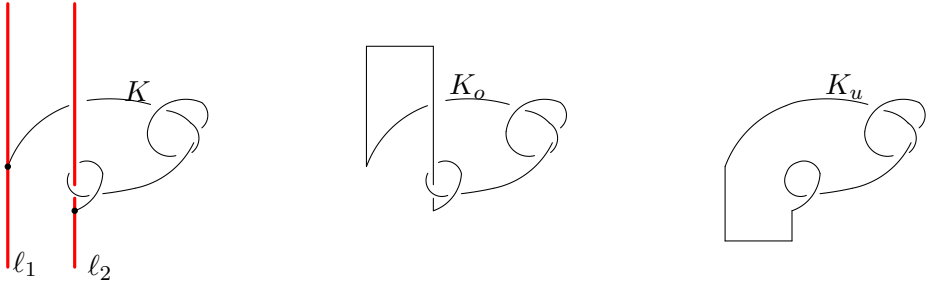

For a given rail knotoid diagram $K$, we have the freedom 
by the above definition to consider many distinct companion loops. Two choices $K_{o1},K_{o2}$ of $K_o$  differ only on their three line segments added to $K$ and by the definition the three segments of $K_{o1}$ can be isotoped on $\pi$ to coincide with those of $K_{o2}$ avoiding the diagram $K$. Similarly for two choices of $K_u$. Thus:

\begin{lemma}\label{lemma0}
 For a given rail knotoid diagram $K$, the companion loops  $K_o,K_u$ are uniquely defined up to isotopy in space.
\end{lemma}

Even stronger, it holds that the isotopy classes of the companion loops $K_{o},K_{u}$ remain invariant for equivalent rail knotoid diagrams:

\begin{lemma}\label{lemma}
	If $K_1,K_2$ are equivalent rail knotoid diagrams, then any choices of $(K_1)_{o},(K_2)_{o}$ are isotopic in space and similarly for $(K_1)_{u},(K_2)_{u}$.
\end{lemma}

\begin{proof}
We will show the  result for two diagrams which differ by a single rail knotoid move. Then the result will follow for all equivalent rail knotoid diagrams as all of them follow from a given one by a finite sequence of rail knotoid moves. We treat the case of the  over companion loops $(K_1)_{o},(K_2)_{o}$, as the case for $(K_1)_{u},(K_2)_{u}$ is treated similarly.

So let the given rail knotoid  diagrams $K_1,K_2$ differ by a single rail knotoid move.  We make some choice of $(K_1)_{o}$  and then we choose $(K_2)_o$  to be the diagram of $(K_1)_{o}$ altered by the rail knotoid move which transforms $K_1$ to $K_2$. By Lemma \ref{lemma0} our choices do not alter the isotopy classes of the two companion loops.

Now, the  rail knotoid move involves either none of the rails or just one of them.  In the first case, the move is a usual Reidemeister move or a planar isotopy move between the knot diagrams   $(K_1)_o$ and $(K_2)_o$ on $\pi$, thus the knots in space which these diagrams represent are isotopic. In the second case, $(K_1)_o$ and $(K_2)_o$ are related to each other   
via the usual Reidemeister moves and planar isotopy moves of knot diagrams as one can easily check:  for example, Figure \ref{figure_slide_invariant} deals with the case of a slide move, and Figure \ref{figure_omega_1_behaviour} deals with an $\Omega_1$ rail knotoid move involving $\ell_1$. Thus $(K_1)_{o},(K_2)_{o}$ are isotopic knots in space as wanted in this case as well. 
\end{proof}

So to each rail knotoid $\kappa$ there correspond $2$ isotopy classes of unoriented  knots, namely the isotopy classes of some choice of the companion loops  $K_o,K_u$. Orienting these knots, there correspond to $\kappa$  four isotopy classes of oriented knots.

\begin{definition} \rm
We call as \textit{over} and \textit{under companion knots} of a rail knotoid $\kappa$,  the isotopy classes  $\kappa_o,\kappa_u$ of a choice of the unoriented loops $K_o,K_u$ where $K$ is any rail knotoid diagram representing $\kappa$. We call the loops $K_o,K_u$  as a choice of companion loops of $\kappa$.

 We denote the oriented loops $K_o,K_u$ as $K_{o+},K_{o-},K_{u+},K_{u-}$ and their corresponding isotopy classes as $\kappa_{o+},\kappa_{o-},\kappa_{u+},\kappa_{o-}$, where the $+$ sign indicates the orientation induced by the rail knotoid $\kappa$. We call $K_{o+},K_{o-},K_{u+},K_{u-}$ as a choice of oriented companion loops of $\kappa$ and $\kappa_{o+},\kappa_{o-},\kappa_{u+},\kappa_{o-}$ as oriented companion  knots of $\kappa$.
 \end{definition}

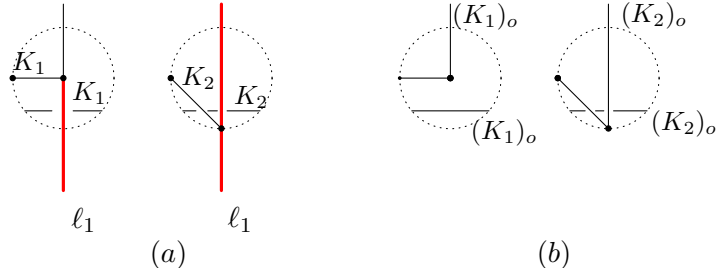
\begin{figure}[!h]
	\centering
\definecolor{ffqqqq}{rgb}{1.,0.,0.}
\begin{tikzpicture}[line cap=round,line join=round,>=triangle 45,x=1.0cm,y=1.0cm]
\clip(0,2.4) rectangle (10.0,6.3);
\draw [line width=1.2pt,color=ffqqqq] (3.356512409852039,5.910589215319042)-- (3.356512409852039,3.4302026165564565);
\draw [line width=0.4pt] (0.5819289219473257,4.926101426790625)-- (1.2554026305918522,4.926101426790625);
\draw [line width=1.2pt,color=ffqqqq] (1.2554026305918522,4.926101426790625)-- (1.2554026305918522,3.4302026165564565);
\draw [line width=0.4pt] (2.6830387012075123,4.926101426790625)-- (3.356512409852039,4.2560095527823165);
\draw [line width=0.4pt] (7.828158022624398,4.926101426790625)-- (8.501631731268924,4.2560095527823165);
\small
\draw (0.45,5.4) node[anchor=north west] {$K_1$};
\draw (1.25,5.0) node[anchor=north west] {$K_1$};
\draw (2.7,5.2) node[anchor=north west] {$K_2$};
\draw (3.4,4.95) node[anchor=north west] {$K_2$};
\draw (6.3,6.05) node[anchor=north west] {$(K_1)_o$};
\draw (6.55,4.5) node[anchor=north west] {$(K_1)_o$};
\draw (8.541883154242406,6.054036111807506) node[anchor=north west] {$(K_2)_o$};
\draw (8.941749096466857,4.623404629626681) node[anchor=north west] {$(K_2)_o$};
\normalsize
\draw (1.2287793664485267,3.3349477046812184) node[anchor=north west] {$\ell_1$};
\draw (3.3080822660156803,3.3438336145084286) node[anchor=north west] {$\ell_1$};
\draw (2.2684308162321036,2.9) node[anchor=north west] {$(a)$};
\draw (7.413372606186727,2.9) node[anchor=north west] {$(b)$};
\draw [line width=0.4pt] (1.2554026305918522,5.910589215319042)-- (1.2554026305918522,4.926101426790625);
\draw [line width=0.4pt] (6.400521952008738,4.926101426790625)-- (6.400521952008738,5.910589215319042);
\draw [line width=0.4pt] (8.501631731268924,4.2560095527823165)-- (8.501631731268924,5.910589215319042);
\draw [line width=0.4pt,dotted] (1.2554026305918522,4.926101426790625) circle (0.6734737086445265cm);
\draw [line width=0.4pt,dotted] (3.356512409852039,4.926101426790625) circle (0.6734737086445266cm);
\draw [line width=0.4pt,dotted] (6.400521952008738,4.926101426790625) circle (0.6734737086445265cm);
\draw [line width=0.4pt,dotted] (8.501631731268922,4.926101426790625) circle (0.6734737086445248cm);
\draw [line width=0.4pt] (2.8429276147012885,4.490443884189358)-- (3.018005344436278,4.490443884189358);
\draw [line width=0.4pt] (3.175574543485869,4.490443884189358)-- (3.302218956275326,4.490443884189358);
\draw [line width=0.4pt] (3.4315449191127763,4.490443884189358)-- (3.8700972050027893,4.490443884189358);
\draw [line width=0.4pt] (5.886937156857993,4.490443884189358)-- (6.914106747159483,4.490443884189358);
\draw [line width=0.4pt] (0.7418178354411022,4.490443884189358)-- (1.1032114251129292,4.490443884189358);
\draw [line width=0.4pt] (1.385283071780357,4.490443884189358)-- (1.7689874257426035,4.490443884189358);
\draw [line width=0.4pt] (5.727048243364212,4.926101426790625)-- (6.400521952008738,4.926101426790625);
\draw [line width=0.4pt] (7.988046936118164,4.490443884189358)-- (8.192414305762304,4.4904438841893555);
\draw [line width=0.4pt] (8.33031947960124,4.4904438841893555)-- (8.440643618672388,4.4904438841893555);
\draw [line width=0.4pt] (8.569355114255394,4.4904438841893555)-- (9.015216526419682,4.490443884189358);
\begin{scriptsize}
\draw [fill=black] (-6.4895721150954095,4.294551717384737) circle (1.0pt);
\draw [fill=black] (1.2554026305918522,4.926101426790625) circle (1.0pt);
\draw [fill=black] (3.356512409852039,4.2560095527823165) circle (1.0pt);
\draw [fill=black] (0.5819289219473257,4.926101426790625) circle (1.0pt);
\draw [fill=black] (2.6830387012075123,4.926101426790625) circle (1.0pt);
\draw [fill=black] (7.828158022624398,4.926101426790625) circle (1.0pt);
\draw [fill=black] (6.400521952008738,4.926101426790625) circle (1.0pt);
\draw [fill=black] (6.400521952008738,4.926101426790625) circle (1.0pt);
\draw [fill=black] (6.400521952008738,4.926101426790625) circle (1.0pt);
\draw [fill=black] (7.828158022624398,4.926101426790625) circle (1.0pt);
\draw [fill=black] (8.501631731268924,4.2560095527823165) circle (1.0pt);
\draw [fill=black] (5.727048243364212,4.926101426790625) circle (0.5pt);
\end{scriptsize}
\end{tikzpicture}
	\caption{(a) Rail knotoid diagrams $K_1,K_2$ connected via a slide move involving rail $\ell_1$. (b) The corresponding over companion loops $(K_1)_{o},(K_2)_{o}$ .}
	\label{figure_slide_invariant}
\end{figure}

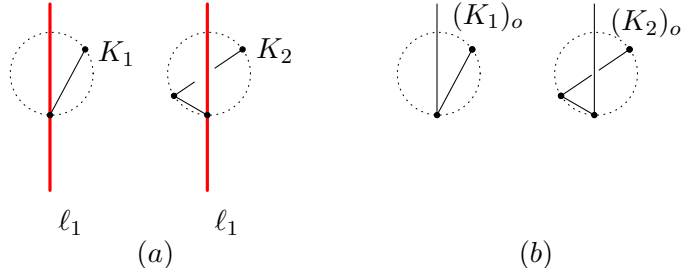
\begin{figure}[!h]
	\centering
\definecolor{ffqqqq}{rgb}{1.,0.,0.}
\begin{tikzpicture}[line cap=round,line join=round,>=triangle 45,x=1.0cm,y=1.0cm]
\clip(0.6526375565961707,2.4) rectangle (9.7,6.0);
\draw [line width=1.2pt,color=ffqqqq] (1.2554026305918522,5.910589215319042)-- (1.2554026305918522,3.4302026165564565);
\draw (1.7565764675305908,5.5474344577555215) node[anchor=north west] {$K_1$};
\draw (3.8615447299055443,5.584856115753299) node[anchor=north west] {$K_2$};
\draw (6.462349960751042,6.071337669724406) node[anchor=north west] {$(K_1)_o$};
\draw (8.567318223125996,5.987138939229407) node[anchor=north west] {$(K_2)_o$};
\draw (1.2326732555617135,3.3) node[anchor=north west] {$\ell_1$};
\draw (3.3095752744383344,3.3) node[anchor=north west] {$\ell_1$};
\draw (2.2711242650000236,2.9) node[anchor=north west] {$(a)$};
\draw (7.3511143381982444,2.9) node[anchor=north west] {$(b)$};
\draw [line width=0.4pt] (1.2554026305918522,4.431971119785855)-- (1.7237121717045791,5.3061351934155265);
\draw [line width=1.2pt,color=ffqqqq] (3.347290063127479,5.910589215319042)-- (3.347290063127479,3.4302026165564565);
\draw [line width=0.4pt] (3.347290063127479,4.431971119785855)-- (2.9044163803237035,4.68828042681962);
\draw [line width=0.4pt] (2.9044163803237035,4.68828042681962)-- (3.174117167331428,4.871159031729778);
\draw [line width=0.4pt] (3.448279832560726,5.057063145291555)-- (3.815599604240206,5.3061351934155265);
\draw [line width=0.4pt] (6.400521952008738,4.431971119785855)-- (6.868831493121465,5.3061351934155265);
\draw [line width=0.4pt] (6.400521952008738,4.431971119785855)-- (6.400521952008738,5.910589215319042);
\draw [line width=0.4pt] (8.554928231202023,5.030976796916886)-- (8.96071892565709,5.3061351934155265);
\draw [line width=0.4pt] (8.492409384544365,5.910589215319042)-- (8.492409384544365,4.431971119785855);
\draw [line width=0.4pt,dotted] (6.423384292134568,4.982247308728974) circle (0.5507509153090827cm);
\draw [line width=0.4pt,dotted] (8.515271724670193,4.982247308728974) circle (0.5507509153090827cm);
\draw [line width=0.4pt,dotted] (3.3701524032533077,4.982247308728974) circle (0.5507509153090829cm);
\draw [line width=0.4pt,dotted] (1.278264970717681,4.982247308728974) circle (0.5507509153090819cm);
\draw [line width=0.4pt] (8.049535701740588,4.68828042681962)-- (8.492409384544365,4.431971119785855);
\draw [line width=0.4pt] (8.049535701740588,4.68828042681962)-- (8.453738375903715,4.962362019135443);
\begin{scriptsize}
\draw [fill=black] (1.2554026305918522,4.431971119785855) circle (1.0pt);
\draw [fill=black] (1.7237121717045791,5.3061351934155265) circle (1.0pt);
\draw [fill=black] (3.347290063127479,4.431971119785855) circle (1.0pt);
\draw [fill=black] (3.815599604240206,5.3061351934155265) circle (1.0pt);
\draw [fill=black] (2.9044163803237035,4.68828042681962) circle (1.0pt);
\draw [fill=black] (6.868831493121465,5.3061351934155265) circle (1.0pt);
\draw [fill=black] (8.96071892565709,5.3061351934155265) circle (1.0pt);
\draw [fill=black] (6.400521952008738,4.431971119785855) circle (1.0pt);
\draw [fill=black] (6.868831493121465,5.3061351934155265) circle (1.0pt);
\draw [fill=black] (8.492409384544365,4.431971119785855) circle (1.0pt);
\draw [fill=black] (8.049535701740588,4.68828042681962) circle (1.0pt);
\draw [fill=black] (8.049535701740588,4.68828042681962) circle (1.0pt);
\draw [fill=black] (8.96071892565709,5.3061351934155265) circle (1.0pt);
\end{scriptsize}
\end{tikzpicture}
	\caption{(a) Rail knotoid diagrams $K_1,K_2$ connected via an $\Omega_1$ move involving rail $\ell_1$. (b) The corresponding  over companion loops $(K_1)_{o},(K_2)_{o}$ .}
	\label{figure_omega_1_behaviour}
\end{figure}

\section{Invariants for rail knotoids}

In this section we define invariants for rail knotoids regarding the equivalence relation of rail knotoid diagrams, as well as their regular equivalence relation which we define below following the corresponding definition of regural isotopy for knot diagrams.

\begin{definition}
	\rm We call two rail knotoid diagrams on the plane $\pi$ as \textit{regularly equivalent} whenever one can be obtained from the other via a finite sequence of all rail moves except for $\Omega_1$ moves, that is, by the Reidemester moves $ \Omega_2, \ \Omega_3$, along with  slide moves and planar isotopy moves on $\pi$.
\end{definition}

The regular equivalence relation is clearly an equivalence relation on the set $\mathcal{RKD}$ of rail knotoid diagrams and any two regularly  equivalent rail diagrams are rail equivalent as well. Now we define as expected:

\begin{definition}
	\rm An \textit{equivalence invariant} or just an \textit{invariant} of rail knotoid diagrams is a map $f:\mathcal{RKD}\rightarrow A$ of the rail knotoid diagrams to some set $A$, so that equivalent rail knotoid diagrams correspond to the same element of $A$. 
	
	There is no harm considering these maps as maps $f:\mathcal{RK}\rightarrow A$ defined on the set $\mathcal{RK}$ of rail knotoids instead of the set of their diagrams, as $f(\kappa)=f(K)$ where $K$ is any rail knotoid diagram of the rail knotoid $\kappa$. We then call $f$ as an \textit{equivalence invariant} or just an \textit{invariant} of rail knotoids. 
	
	\noindent Similarly,  a \textit{regular equivalence invariant} or just a \textit{regular invariant} of rail knotoid diagrams is a map $\mathcal{RKD}\rightarrow A$ of the rail knotoid diagrams to some set $A$, so that regularly equivalent rail knotoid diagrams correspond to the same element of $A$.	By abuse of language we can say that $f$ is a regular invariant of rail knotoids as well.
\end{definition}

Under the above definition, Lemma \ref{lemma} translates to:

\begin{prop}
	The isotopy classes of the companion loops of the rail knotoid diagrams, are invariants of the rail knotoids.
\end{prop}

An immediate consequence of this Proposition is the first part of the Theorem that follows. The second part of the Theorem is an immediate consequence of the arguments in the proof of Lemma \ref{lemma}, which reveal that two regularly equivalent rail knotoid diagrams have regularly equivalent companion loops.

\begin{theorem}
	Every ambient isotopy invariant $f:\mathcal{K}\rightarrow A$ of unoriented knots provides $2$ invariants $f_o,f_u:\mathcal{RK}\rightarrow A$ of rail knotoids via the corresponding unoriented companion loops as $f_{o}(\kappa)=f(\kappa_{o}), f_{u}(\kappa)=f(\kappa_{u})$ and every ambient isotopy invariant $f$ of oriented knots $f:\mathcal{K}^{or}\rightarrow A$ provides $4$ invariants $f_{o-},f_{o+},f_{u-},f_{u-}:\mathcal{RK}\rightarrow A$ of rail knotoids via the corresponding oriented companion knots, as $f_{o\pm}(\kappa)=f(\kappa_{o\pm}),f_{u\pm}(\kappa)=f(\kappa_{u\pm})$.
	
	Similarly, every regular isotopy invariant of unoriented knot diagrams $f:\mathcal{KD}\rightarrow A$ provides $2$ regular equivalence invariants $f:\mathcal{RKD}\rightarrow A$ of rail knotoid diagrams via the corresponding unoriented companion loops as $f_{o}(K)=f(K_{o}), f_{u}(K)=f(K_{u})$, and every regular isotopy invariant $f:\mathcal{KD}^{or}\rightarrow A$ of oriented knot diagrams provides $4$ regular equivalence invariants of rail knotoid diagrams via the corresponding oriented companion loops as $f_{o\pm}(K)=f(K_{o\pm}),f_{u\pm}(K)=f(K_{u\pm})$.
	
\end{theorem} 

We call the above as \textit{knot type invariants}, or just \textit{knot invariants} for the rail knotoids. If we  care to consider only the orientation induced by the rail knotoid for the companion loops, then all oriented invariants are reduced in number by a factor of $2$. 

In general, we expect the invariants of the Theorem to be distinct. As an example of how the Theorem applies, we give a  list of well-defined knot invariants for rail knotoids:

\begin{definition}\rm We define $4$ versions for each of the normalized bracket, the Jones and the HOMFLYPT polynomials, each an oriented equivalence invariant of rail knotoids, as follows:
	\[
\begin{array}{rcl}
\mathbb{X}_{o*}(\kappa)&:=& \mathbb{X}(\kappa_{o*})\vspace{1ex}\\
\mathbb{X}_{u*}(\kappa)&:=& \mathbb{X}(\kappa_{u*})
\vspace{1ex}\\
\mathbb{J}_{o*}(\kappa)&:=& \mathbb{J}( \kappa_{o*})
\vspace{1ex}\\
\mathbb{J}_{u*}(\kappa)&:=& \mathbb{J}( \kappa_{u*})
\vspace{1ex}\\
\mathbb{H}_{o*}(\kappa)&:=& \mathbb{H}( \kappa_{o*})
\vspace{1ex}\\
\mathbb{H}_{u*}(\kappa)&:=& \mathbb{H}( \kappa_{u*})\vspace{1ex}\\
\end{array}
\]

\noindent where $*\in\{+,-\} $ and on the right hand sides we have the usual normalized bracket, the Jones and the HOMFLYPT  polynomials for oriented knots.

We also define $2$ versions of the Kauffman polynomial, each an unoriented regular equivalence invariant of rail knotoids, as follows:
	\[
\begin{array}{rcl}
\mathbb{K}_{o}(\kappa)&:=& \mathbb{K}( \kappa_{o}) 
\vspace{1ex}\\
\mathbb{K}_{u}(\kappa)&:=& \mathbb{K}(\kappa_{u})
\end{array}
\]
\noindent where on the right hand sides we have the usual Kauffman polynomial for knots.

\end{definition}

The arguments in the proof of Lemma \ref{lemma}, imply that we can also define bracket polynomials for rail knotoid diagram which are regular equivalence invariants:

\begin{definition}\rm
For the rail knotoid diagram $K$ we define $2$ vesions of the bracket polynomial, each a regular equivalence invariant of rail knotoid diagrams:

\[
\begin{array}{rcl}
\left\langle K \right\rangle_o &:=&\left\langle K_o \right\rangle \vspace{1ex}\\
\left\langle K \right\rangle_u &:=&\left\langle K_u \right\rangle \end{array}
\]

\noindent where  on the right hand sides we have the usual bracket polynomial for a choice of a companion loop $K_o$ or $K_u$ of $K$. 
\end{definition}

Under a rail knotoid $\Omega_1$ move between rail knotoid diagrams, the rail bracket  polynomial exhibits the same behavior as the usual bracket polynomial, and denoting as usual by $w$ the writhe of oriented knots, it follows that the rail bracket polynomial is related to the rail normalized bracket  and rail Jones  polynomials  as described in the following Theorem:

\begin{theorem}
	
For any rail knotoid $\kappa$ it holds:

\[
\begin{array}{rcl}
\mathbb{X}_{o*}(\kappa)&=& (-A^3)^{-w(K_{o*})} \langle K_{o} \rangle 
\vspace{1ex}\\
\mathbb{X}_{u*}(\kappa)&=& (-A^3)^{-w(K_{u*})} \langle K_{u} \rangle \vspace{1ex}\\
\mathbb{J}_{o*}(\kappa)&=& (-t^{-3/4})^{-w(K_{o*})} \langle K_{o} \rangle  
\vspace{1ex}\\
\mathbb{J}_{u*}(\kappa)&=& (-t^{-3/4})^{-w(K_{u*})} \langle K_{u} \rangle   \\
\end{array}
\]

\noindent where $*\in\{+,-\} $, $K_{o},K_{u}$ are choices of the companion loops and $K_{o*},K_{u*}$ are choices of the oriented companion loops of  any rail knotoid diagram $K$ of the rail knotoid $\kappa$.
\end{theorem}

\section{Connections of rail knotoids to rail arcs}

As mentioned in the Introduction, rail knotoids were defined in \cite{KL} in order to study the isotopies of rail arcs which were introduced in \cite{GK,KL}. The connection between the two is given in the Theorem of \S \ref{section_rail_knotoids}, namely, two rail arcs are rail isotopic in space if and only if their rail knotoid diagrams are equivalent, that is, if and only if the two rail arcs correspond to the same rail knotoid.

In the light of the previous section we have that the  definition below refers to well-defined notions and that the Theorem following immediately after holds automatically:

\begin{definition} \rm
		To each isotopy class $c$ of rail arcs, there correspond $2$ isotopy classes of unoriented  knots, namely, those of the companion loops of the corresponding to $c$ rail knotoid $\kappa_c$. There also correspond to $c$, $4$  isotopy classes of oriented  knots, namely, those of the oriented companion loops of $\kappa_c$.
		
		Let us denote the above as $c_{a},c_{ab}$ for $a\in\{o,u\}, \ b\in\{+,-\}$ with the obvious correspondence of the indices with the various kinds of isotopy classes of knots.
\end{definition}

\begin{theorem}\label{theorem_rail_isotopy}	
	If $f$ is an unoriented or oriented knot invariant, then there exists an invariant of the isotopy classes of rail arcs defined as $f_{a}(c)=f(c_a)$ or $ f_{ab}(c)=f(c_{ab})$, depending on the kind of knots concerning $f$.
\end{theorem}

\section{Connections of rail knotoids to knotoids}

For the rail arcs in space, there has been at least one more effort in as much as we know, to study their rail isotopy diagrammatically  \cite{GK}. This effort uses the tool of the usual planar knotoids, whereupon one projects the rail arc on a plane $p$ perpendicular to the rails resulting in a planar knotoid diagram and then one translates the rail isotopy of arcs to  equivalence of knotoid diagrams, thus to knotoids, as the main result of this approach goes.

The difference of the method in \cite{GK} to ours concerns the choice of the projection plane and it is at least conceptual. Figure \ref{figure_comparison} makes it clear that the original rail arc is encoded in significantly different ways. It is interesting to note that the triad of sets $A$=(rail arc, rails, $p$= projection plane perpendicular to the rails) which is of interest in the case of knotoids, is not  isotopic in space to $B$=(rail arc, rails, $\pi$= plane of the rails)  which is of interest in the case of rail knotoids (see Figure \ref{figure_rail_arcs}). Nevertheless this is not entirely unexpected, as for example none of $A,B$ remains ambient isotopic invariant under the various rail isotopic choices of the rail arc $c$.  

\begin{figure}
	\centering
\definecolor{ffqqqq}{rgb}{1.,0.,0.}
\begin{tikzpicture}[line cap=round,line join=round,>=triangle 45,x=1.0cm,y=1.0cm]
clip(-0.5684768639573702,-2.959970080927272) rectangle (9.574207584995872,5.229989162678216);
\draw [shift={(1.9599846506341152,-4.621212329830088)},line width=0.4pt]  plot[domain=1.445942772822315:1.6443585287963023,variable=\t]({1.*8.945852281313767*cos(\t r)+0.*8.945852281313767*sin(\t r)},{0.*8.945852281313767*cos(\t r)+1.*8.945852281313767*sin(\t r)});
\draw [shift={(3.66123781920074,3.754244689349341)},line width=0.4pt]  plot[domain=-1.0005806608179686:1.0369087230177279,variable=\t]({1.*0.33706815578405974*cos(\t r)+0.*0.33706815578405974*sin(\t r)},{0.*0.33706815578405974*cos(\t r)+1.*0.33706815578405974*sin(\t r)});
\draw [shift={(-0.04144855591274061,10.004814604161622)},line width=0.4pt]  plot[domain=4.911642736258371:5.248753851412097,variable=\t]({1.*7.601816934580359*cos(\t r)+0.*7.601816934580359*sin(\t r)},{0.*7.601816934580359*cos(\t r)+1.*7.601816934580359*sin(\t r)});
\draw [shift={(1.5803359560938772,2.1402199892244806)},line width=0.4pt]  plot[domain=2.689536141245218:4.2786253653758815,variable=\t]({1.*0.4882276292562123*cos(\t r)+0.*0.4882276292562123*sin(\t r)},{0.*0.4882276292562123*cos(\t r)+1.*0.4882276292562123*sin(\t r)});
\draw [shift={(2.6491329133262878,5.344668011443012)},line width=0.4pt]  plot[domain=4.376353683400649:4.8558399011415645,variable=\t]({1.*3.863551724386354*cos(\t r)+0.*3.863551724386354*sin(\t r)},{0.*3.863551724386354*cos(\t r)+1.*3.863551724386354*sin(\t r)});
\draw [line width=0.4pt,color=ffqqqq] (1.3025014164853015,4.965913163119643)-- (1.3025014164853015,1.9773817571705958);
\draw [line width=0.4pt,color=ffqqqq] (3.2014640806820824,4.9970436985982785)-- (3.2014640806820824,3.419763234347393);
\draw [shift={(7.423756024206408,-4.6212123298301)},line width=0.4pt]  plot[domain=1.445942772822315:1.6443585287963023,variable=\t]({1.*8.94585228131378*cos(\t r)+0.*8.94585228131378*sin(\t r)},{0.*8.94585228131378*cos(\t r)+1.*8.94585228131378*sin(\t r)});
\draw [shift={(9.125009192773033,3.7542446893493424)},line width=0.4pt]  plot[domain=-1.0005806608179721:1.036908723017727,variable=\t]({1.*0.33706815578406063*cos(\t r)+0.*0.33706815578406063*sin(\t r)},{0.*0.33706815578406063*cos(\t r)+1.*0.33706815578406063*sin(\t r)});
\draw [shift={(5.422322817659544,10.00481460416164)},line width=0.4pt]  plot[domain=4.911642736258372:5.248753851412097,variable=\t]({1.*7.601816934580379*cos(\t r)+0.*7.601816934580379*sin(\t r)},{0.*7.601816934580379*cos(\t r)+1.*7.601816934580379*sin(\t r)});
\draw [shift={(7.0441073296661685,2.14021998922448)},line width=0.4pt]  plot[domain=2.689536141245216:4.278625365375883,variable=\t]({1.*0.48822762925621105*cos(\t r)+0.*0.48822762925621105*sin(\t r)},{0.*0.48822762925621105*cos(\t r)+1.*0.48822762925621105*sin(\t r)});
\draw [shift={(8.112904286898578,5.344668011443001)},line width=0.4pt]  plot[domain=4.376353683400649:4.855839901141565,variable=\t]({1.*3.8635517243863426*cos(\t r)+0.*3.8635517243863426*sin(\t r)},{0.*3.8635517243863426*cos(\t r)+1.*3.8635517243863426*sin(\t r)});
\draw [line width=0.4pt,color=ffqqqq] (8.665235454254375,3.004689427965581)-- (8.665235454254375,-1.0159892551622187);
\draw [line width=0.4pt,color=ffqqqq] (6.766272790057594,1.5311774153101485)-- (6.766272790057594,-1.0297777949224782);
\draw [line width=0.4pt,color=ffqqqq] (6.766272790057594,4.611004050583099)-- (6.766272790057594,4.965913163119643);
\draw [line width=0.4pt,color=ffqqqq] (6.766272790057594,4.611004050583099)-- (6.766272790057594,1.9773817571705958);
\draw [line width=0.4pt,color=ffqqqq] (8.665235454254375,4.9970436985982785)-- (8.665235454254375,3.419763234347393);
\draw [line width=0.4pt] (0.19259183453671883,1.0936573281575455)-- (-0.1659101992300375,-0.14731125026584502);
\draw [line width=0.4pt] (-0.1659101992300375,-0.14731125026584502)-- (4.439462080696756,-0.10594563098506533);
\draw [line width=0.4pt] (4.797964114463512,1.135022947438325)-- (4.439462080696756,-0.10594563098506533);
\draw [shift={(2.8891249934900234,2.0551960992755487)},line width=0.4pt]  plot[domain=1.2174638286973338:1.4007018200900474,variable=\t]({1.*2.210292958581322*cos(\t r)+0.*2.210292958581322*sin(\t r)},{0.*2.210292958581322*cos(\t r)+1.*2.210292958581322*sin(\t r)});
\draw [shift={(2.8130060995137156,-1.8543677381885584)},line width=0.4pt]  plot[domain=1.5094716484677242:2.157451202042953,variable=\t]({1.*2.72861967660218*cos(\t r)+0.*2.72861967660218*sin(\t r)},{0.*2.72861967660218*cos(\t r)+1.*2.72861967660218*sin(\t r)});
\draw [shift={(3.112705263764707,-1.089784375935884)},line width=0.4pt]  plot[domain=1.2925435653448993:1.4793350242080112,variable=\t]({1.*1.9471715413703072*cos(\t r)+0.*1.9471715413703072*sin(\t r)},{0.*1.9471715413703072*cos(\t r)+1.*1.9471715413703072*sin(\t r)});
\draw [shift={(3.7069467412406834,0.5063607841023098)},line width=0.4pt]  plot[domain=-0.9697168301771075:1.2933916926539513,variable=\t]({1.*0.23882349117614377*cos(\t r)+0.*0.23882349117614377*sin(\t r)},{0.*0.23882349117614377*cos(\t r)+1.*0.23882349117614377*sin(\t r)});
\draw [shift={(3.1841098107963224,2.0117618976527494)},line width=0.4pt]  plot[domain=3.999016798890155:5.081170544213329,variable=\t]({1.*1.825069378983462*cos(\t r)+0.*1.825069378983462*sin(\t r)},{0.*1.825069378983462*cos(\t r)+1.*1.825069378983462*sin(\t r)});
\draw [shift={(1.6089140266861097,0.3380057646664064)},line width=0.4pt]  plot[domain=1.0791871989036625:2.09033193549408,variable=\t]({1.*0.49837198902584917*cos(\t r)+0.*0.49837198902584917*sin(\t r)},{0.*0.49837198902584917*cos(\t r)+1.*0.49837198902584917*sin(\t r)});
\draw [shift={(1.4087178831559033,0.47781122527062336)},line width=0.4pt]  plot[domain=2.3491280499025846:4.091374488103748,variable=\t]({1.*0.28601237319440026*cos(\t r)+0.*0.28601237319440026*sin(\t r)},{0.*0.28601237319440026*cos(\t r)+1.*0.28601237319440026*sin(\t r)});
\draw [shift={(2.167979623071038,2.4472894570924932)},line width=0.4pt]  plot[domain=4.314450767782321:4.994553654522356,variable=\t]({1.*2.3887403841583574*cos(\t r)+0.*2.3887403841583574*sin(\t r)},{0.*2.3887403841583574*cos(\t r)+1.*2.3887403841583574*sin(\t r)});
\draw [shift={(2.4024638069877735,1.257003072012459)},line width=0.4pt]  plot[domain=5.222878857972777:5.473381552426566,variable=\t]({1.*1.1585749487927424*cos(\t r)+0.*1.1585749487927424*sin(\t r)},{0.*1.1585749487927424*cos(\t r)+1.*1.1585749487927424*sin(\t r)});
\draw [line width=0.8pt,dotted,color=ffqqqq] (1.3025014164853015,0.10221455338024787)-- (1.3025014164853015,-0.050018938224290155);
\draw [line width=0.8pt,color=ffqqqq] (1.3025014164853015,-0.30612124725953915)-- (1.3025014164853015,-1.0297777949224782);
\draw [line width=0.8pt,dotted,color=ffqqqq] (3.2014640806820833,0.1022145533802473)-- (3.2014640806820833,-0.050018938224290745);
\draw [line width=0.8pt,color=ffqqqq] (3.2014640806820833,-0.30612124725953976)-- (3.2014640806820824,-1.0159892551622187);
\draw [line width=0.4pt] (0.19259183453671883,1.0936573281575455)-- (0.4351553799932838,1.0958360426376943);
\draw [line width=0.4pt] (0.6999844592141737,1.0982147469420735)-- (1.132622387830432,1.1021007163607823);
\draw [line width=0.4pt] (1.4618996531916464,1.1050582965885778)-- (3.024294567646276,1.1190917838441583);
\draw [line width=0.4pt] (3.3471734078087962,1.1219918931869355)-- (3.5926912278952576,1.1241971430679516);
\draw [line width=0.4pt] (3.8250062552001234,1.1262838049898516)-- (4.797964114463512,1.135022947438325);
\draw [line width=0.4pt,dotted] (0.5924594875842547,0.41801887990481035)-- (0.5924594875842547,-0.045071926734751244);
\draw [line width=0.4pt] (3.708669473402983,3.1438450090904495)-- (3.708669473402983,0.41801887990481035);
\draw [line width=0.4pt,dotted] (3.708669473402983,0.18623239719038231)-- (3.7086694734029835,-0.045071926734751244);
\draw [line width=0.8pt,color=ffqqqq] (1.3025014164853015,1.5311774153101485)-- (1.3025014164853017,0.4180188799048104);
\draw [line width=0.8pt,color=ffqqqq] (1.3025014164853017,0.4180188799048104)-- (1.3025014164853015,0.4038511804844791);
\draw [line width=0.8pt,color=ffqqqq] (3.2014640806820824,3.004689427965581)-- (3.201464080682083,0.41801887990481046);
\draw [line width=0.8pt,dotted,color=ffqqqq] (3.201464080682083,0.41801887990481046)-- (3.2014640806820824,0.4038511804844791);
\draw (1.9700659029946872,4.75) node[anchor=north west] {$c$};
\draw (7.575535904228189,4.8) node[anchor=north west] {$K_2$};
\draw (0.4,-2.1214469045132462) node[anchor=north west] {$K_1$};
\draw (0.55,4.977283547868509) node[anchor=north west] {$\pi$};
\draw (-0.04009239662798266,0.39412262472883197) node[anchor=north west] {$p$};
\draw [shift={(8.727403563026904,3.2170611681163175)},line width=0.4pt]  plot[domain=0.9682337770752321:1.5117444492855252,variable=\t]({1.*1.0041968022010386*cos(\t r)+0.*1.0041968022010386*sin(\t r)},{0.*1.0041968022010386*cos(\t r)+1.*1.0041968022010386*sin(\t r)});
\draw [shift={(2.813006099513715,-4.695562990038753)},line width=0.4pt]  plot[domain=1.509471648467724:2.1574512020429535,variable=\t]({1.*2.7286196766021793*cos(\t r)+0.*2.7286196766021793*sin(\t r)},{0.*2.7286196766021793*cos(\t r)+1.*2.7286196766021793*sin(\t r)});
\draw [shift={(3.7069467412406842,-2.334834467747886)},line width=0.4pt]  plot[domain=-0.9697168301771093:1.2933916926539548,variable=\t]({1.*0.23882349117614357*cos(\t r)+0.*0.23882349117614357*sin(\t r)},{0.*0.23882349117614357*cos(\t r)+1.*0.23882349117614357*sin(\t r)});
\draw [shift={(1.3441997514729447,-2.401431083340495)},line width=0.4pt]  plot[domain=2.3776805876819633:4.252847116180771,variable=\t]({1.*0.28160720288304886*cos(\t r)+0.*0.28160720288304886*sin(\t r)},{0.*0.28160720288304886*cos(\t r)+1.*0.28160720288304886*sin(\t r)});
\draw [shift={(1.5317192287585515,-2.6247406566243368)},line width=0.4pt]  plot[domain=1.174241818365235:2.322543169791534,variable=\t]({1.*0.5723662570772621*cos(\t r)+0.*0.5723662570772621*sin(\t r)},{0.*0.5723662570772621*cos(\t r)+1.*0.5723662570772621*sin(\t r)});
\draw [shift={(2.201047856060525,-0.40117226178152243)},line width=0.4pt]  plot[domain=4.30138907627652:5.013408134076563,variable=\t]({1.*2.4572895775595325*cos(\t r)+0.*2.4572895775595325*sin(\t r)},{0.*2.4572895775595325*cos(\t r)+1.*2.4572895775595325*sin(\t r)});
\draw [shift={(2.039152725599072,-1.418154520543515)},line width=0.4pt]  plot[domain=5.41435756126732:5.581216726296956,variable=\t]({1.*1.5977233342467485*cos(\t r)+0.*1.5977233342467485*sin(\t r)},{0.*1.5977233342467485*cos(\t r)+1.*1.5977233342467485*sin(\t r)});
\draw [shift={(3.1635809821473897,-0.6892385568145063)},line width=0.4pt]  plot[domain=4.051508183608188:4.85141331561854,variable=\t]({1.*1.991646183681109*cos(\t r)+0.*1.991646183681109*sin(\t r)},{0.*1.991646183681109*cos(\t r)+1.*1.991646183681109*sin(\t r)});
\draw [shift={(3.2705171091832304,-1.4493507465833924)},line width=0.4pt]  plot[domain=4.850947354803074:5.198155897874644,variable=\t]({1.*1.224049058236359*cos(\t r)+0.*1.224049058236359*sin(\t r)},{0.*1.224049058236359*cos(\t r)+1.*1.224049058236359*sin(\t r)});
\draw [shift={(2.9835978333101596,-4.376195073170935)},line width=0.4pt]  plot[domain=1.2365224905943704:1.5721959513988593,variable=\t]({1.*2.4041249401903704*cos(\t r)+0.*2.4041249401903704*sin(\t r)},{0.*2.4041249401903704*cos(\t r)+1.*2.4041249401903704*sin(\t r)});
\draw [line width=0.4pt] (0.5924594875842547,0.41801887990481035)-- (0.5924594875842547,5.080663117563389);
\draw [line width=0.4pt] (0.5924594875842547,5.080663117563389)-- (3.7086694734029835,5.080663117563389);
\draw [line width=0.4pt] (3.7086694734029835,5.080663117563389)-- (3.708669473402983,3.586117686102134);
\draw [line width=0.4pt] (0.5924594875842547,-0.30612124725953915)-- (0.5924594875842547,-1.1106244452744203);
\draw [line width=0.4pt] (0.5924594875842547,-1.1106244452744203)-- (3.7086694734029835,-1.1106244452744203);
\draw [line width=0.4pt] (3.7086694734029835,-1.1106244452744203)-- (3.7086694734029835,-0.30612124725953915);
\begin{scriptsize}
\draw [fill=black] (1.3025014164853015,4.300446083038002) circle (1.0pt);
\draw [fill=black] (3.2014640806820824,1.5208005701506027) circle (1.0pt);
\draw [fill=black] (6.766272790057594,4.300446083038002) circle (1.0pt);
\draw [fill=black] (8.665235454254375,1.5208005701506027) circle (1.0pt);
\draw [fill=black] (3.201464080682083,0.41801887990481046) circle (1.0pt);
\draw [fill=black] (1.3025014164853015,0.4038511804844791) circle (1.0pt);
\draw [fill=black] (1.3025014164853017,-2.4231763719453854) circle (1.0pt);
\draw [fill=black] (3.2591303445803397,-2.449839773894238) circle (1.0pt);
\end{scriptsize}
\end{tikzpicture}
	\caption{A rail arc $c$ and its corresponding knotoid and rail knotoid diagrams $K_1$ and $K_2$ respectively.}
\label{figure_comparison}	
\end{figure}
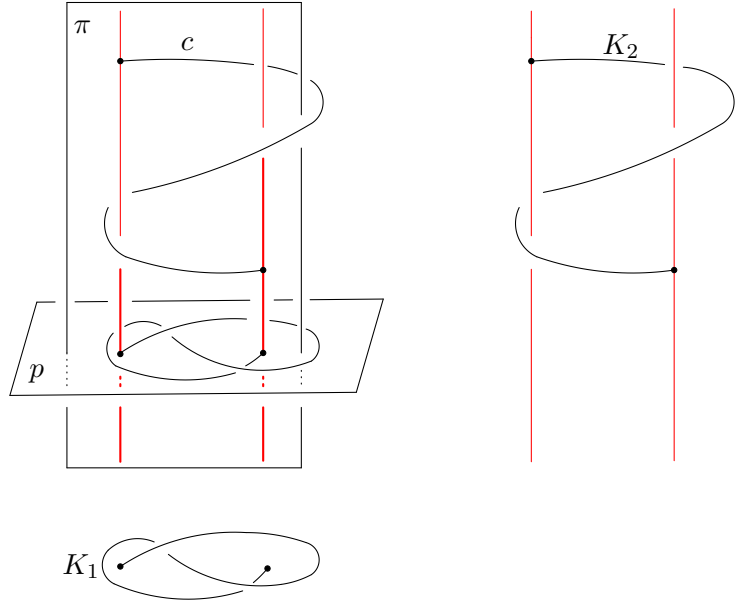

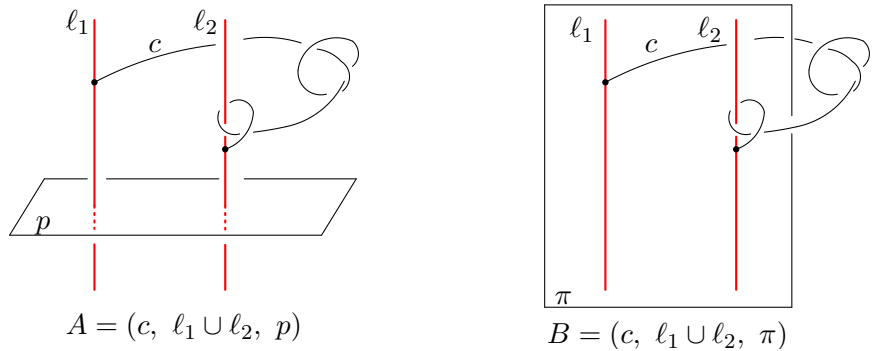
\begin{figure}
	\centering
	\definecolor{ffqqqq}{rgb}{1.,0.,0.}
	\begin{tikzpicture}[line cap=round,line join=round,>=triangle 45,x=1.0cm,y=1.0cm]
	\clip(1.5630270952355263,2.7) rectangle (13.5,7.344936248563249);
\draw [line width=0.4pt] (1.888869779206626,4.211336883714576)-- (6.035784843388968,4.211336883714576);
\draw [line width=0.4pt] (1.888869779206626,4.211336883714576)-- (2.3507153645204215,4.9608895549615655);
\draw [line width=0.4pt] (6.497630428702763,4.9608895549615655)-- (6.035784843388968,4.211336883714576);
\draw [line width=0.8pt,color=ffqqqq] (9.809720865599461,7.073265264839444)-- (9.809720865599461,3.4844979297781022);
\draw [line width=0.4pt] (9.002412949132255,7.275490406286499)-- (9.002412949132255,3.2525870158093584);
\draw [line width=0.4pt] (9.002412949132255,7.275490406286499)-- (12.288147568676091,7.275490406286499);
\draw [line width=0.4pt] (12.288147568676091,3.2525870158093584)-- (9.002412949132255,3.2525870158093584);
\draw [shift={(5.1424548334841,2.238416134351101)},line width=0.4pt]  plot[domain=1.6853446623248363:2.0584873121090985,variable=\t]({1.*4.535911143695885*cos(\t r)+0.*4.535911143695885*sin(\t r)},{0.*4.535911143695885*cos(\t r)+1.*4.535911143695885*sin(\t r)});
\draw [shift={(6.165968007865688,6.320138130917781)},line width=0.4pt]  plot[domain=-0.9090662575001538:0.9923514156973319,variable=\t]({1.*0.24226291935302283*cos(\t r)+0.*0.24226291935302283*sin(\t r)},{0.*0.24226291935302283*cos(\t r)+1.*0.24226291935302283*sin(\t r)});
\draw [shift={(5.282348238349789,5.2300418232257755)},line width=0.4pt]  plot[domain=1.2669866701075727:1.7457523046311114,variable=\t]({1.*1.6130549761659305*cos(\t r)+0.*1.6130549761659305*sin(\t r)},{0.*1.6130549761659305*cos(\t r)+1.*1.6130549761659305*sin(\t r)});
\draw [shift={(6.204868053391834,6.34658325187964)},line width=0.4pt]  plot[domain=1.8031194837987232:3.0144356198078617,variable=\t]({1.*0.48491668043597475*cos(\t r)+0.*0.48491668043597475*sin(\t r)},{0.*0.48491668043597475*cos(\t r)+1.*0.48491668043597475*sin(\t r)});
\draw [shift={(6.309704444107993,6.62558684206024)},line width=0.4pt]  plot[domain=-0.8893986781627756:0.8893986781627703,variable=\t]({1.*0.21663861102863105*cos(\t r)+0.*0.21663861102863105*sin(\t r)},{0.*0.21663861102863105*cos(\t r)+1.*0.21663861102863105*sin(\t r)});
\draw [shift={(6.02638382590094,6.382174114055326)},line width=0.4pt]  plot[domain=3.0561739652779294:4.962142713419938,variable=\t]({1.*0.3036244470810311*cos(\t r)+0.*0.3036244470810311*sin(\t r)},{0.*0.3036244470810311*cos(\t r)+1.*0.3036244470810311*sin(\t r)});
\draw [shift={(5.913020909225199,6.165646600533312)},line width=0.4pt]  plot[domain=0.7476423546557244:1.4462476259445392,variable=\t]({1.*0.5255705655408514*cos(\t r)+0.*0.5255705655408514*sin(\t r)},{0.*0.5255705655408514*cos(\t r)+1.*0.5255705655408514*sin(\t r)});
\draw [shift={(6.234937330459185,6.3080246633316115)},line width=0.4pt]  plot[domain=1.1606768062067996:1.8416066999258576,variable=\t]({1.*0.5297547816122091*cos(\t r)+0.*0.5297547816122091*sin(\t r)},{0.*0.5297547816122091*cos(\t r)+1.*0.5297547816122091*sin(\t r)});
\draw [shift={(11.935191293454709,2.2384161343511146)},line width=0.4pt]  plot[domain=1.6853446623248356:2.058487312109099,variable=\t]({1.*4.53591114369587*cos(\t r)+0.*4.53591114369587*sin(\t r)},{0.*4.53591114369587*cos(\t r)+1.*4.53591114369587*sin(\t r)});
\draw [shift={(12.819120285871552,6.382174114055323)},line width=0.4pt]  plot[domain=3.0561739652779205:4.96214271341994,variable=\t]({1.*0.3036244470810285*cos(\t r)+0.*0.3036244470810285*sin(\t r)},{0.*0.3036244470810285*cos(\t r)+1.*0.3036244470810285*sin(\t r)});
\draw [shift={(12.997604513362443,6.346583251879643)},line width=0.4pt]  plot[domain=1.803119483798718:3.014435619807868,variable=\t]({1.*0.4849166804359695*cos(\t r)+0.*0.4849166804359695*sin(\t r)},{0.*0.4849166804359695*cos(\t r)+1.*0.4849166804359695*sin(\t r)});
\draw [shift={(13.027673790429798,6.3080246633316115)},line width=0.4pt]  plot[domain=1.1606768062067996:1.8416066999258591,variable=\t]({1.*0.5297547816122093*cos(\t r)+0.*0.5297547816122093*sin(\t r)},{0.*0.5297547816122093*cos(\t r)+1.*0.5297547816122093*sin(\t r)});
\draw [shift={(13.1024409040786,6.62558684206024)},line width=0.4pt]  plot[domain=-0.8893986781627534:0.889398678162748,variable=\t]({1.*0.2166386110286336*cos(\t r)+0.*0.2166386110286336*sin(\t r)},{0.*0.2166386110286336*cos(\t r)+1.*0.2166386110286336*sin(\t r)});
\draw [shift={(12.705757369195812,6.165646600533317)},line width=0.4pt]  plot[domain=0.7476423546557193:1.4462476259445411,variable=\t]({1.*0.5255705655408455*cos(\t r)+0.*0.5255705655408455*sin(\t r)},{0.*0.5255705655408455*cos(\t r)+1.*0.5255705655408455*sin(\t r)});
\draw [shift={(12.958704467836304,6.320138130917781)},line width=0.4pt]  plot[domain=-0.9090662575001653:0.9923514156973472,variable=\t]({1.*0.24226291935302136*cos(\t r)+0.*0.24226291935302136*sin(\t r)},{0.*0.24226291935302136*cos(\t r)+1.*0.24226291935302136*sin(\t r)});
\draw [shift={(12.075084698320403,5.230041823225804)},line width=0.4pt]  plot[domain=1.5063319342325:1.7457523046311152,variable=\t]({1.*1.6130549761659023*cos(\t r)+0.*1.6130549761659023*sin(\t r)},{0.*1.6130549761659023*cos(\t r)+1.*1.6130549761659023*sin(\t r)});
\draw [shift={(12.075084698320461,5.230041823226031)},line width=0.4pt]  plot[domain=1.266986670107561:1.4133022627014413,variable=\t]({1.*1.6130549761656683*cos(\t r)+0.*1.6130549761656683*sin(\t r)},{0.*1.6130549761656683*cos(\t r)+1.*1.6130549761656683*sin(\t r)});
\draw [line width=0.4pt] (12.288147568676091,7.275490406286499)-- (12.288147568676091,5.801711532237226);
\draw [line width=0.4pt] (12.288147568676091,5.561173323298162)-- (12.288147568676091,3.2525870158093584);
\draw [line width=0.8pt,color=ffqqqq] (3.016984405628848,4.090197058058497)-- (3.016984405628848,3.4844979297781022);
\draw [line width=0.8pt,color=ffqqqq] (4.755638769444044,4.082625818954993)-- (4.755638769444044,3.4844979297781022);
\draw [line width=0.4pt] (2.3507153645204215,4.9608895549615655)-- (2.9224262249287674,4.9608895549615655);
\draw [line width=0.4pt] (3.1768707777527214,4.9608895549615655)-- (4.670706539493355,4.9608895549615655);
\draw [line width=0.4pt] (5.007229980325036,4.9608895549615655)-- (6.497630428702763,4.9608895549615655);
\draw [line width=0.8pt,color=ffqqqq] (3.016984405628848,7.073265264839444)-- (3.016984405628848,4.605041317096833);
\draw [line width=0.8pt,dotted,color=ffqqqq] (3.016984405628848,4.605041317096833)-- (3.016984405628848,4.294620513853131);
\draw [line width=0.8pt,dotted,color=ffqqqq] (4.755638769444044,4.605041317096833)-- (4.755638769444044,4.287049274749626);
\draw [shift={(4.576372245136618,5.879118414956251)},line width=0.4pt]  plot[domain=5.0417128980031976:6.228110483937388,variable=\t]({1.*0.5543125119013298*cos(\t r)+0.*0.5543125119013298*sin(\t r)},{0.*0.5543125119013298*cos(\t r)+1.*0.5543125119013298*sin(\t r)});
\draw [shift={(4.934148858791462,5.812144430168311)},line width=0.4pt]  plot[domain=0.1842017557222185:2.1870758851406187,variable=\t]({1.*0.19906302498268078*cos(\t r)+0.*0.19906302498268078*sin(\t r)},{0.*0.19906302498268078*cos(\t r)+1.*0.19906302498268078*sin(\t r)});
\draw [shift={(4.883954842814907,5.775748061316205)},line width=0.4pt]  plot[domain=2.734001129563779:4.9986122064633385,variable=\t]({1.*0.2261725957012447*cos(\t r)+0.*0.2261725957012447*sin(\t r)},{0.*0.2261725957012447*cos(\t r)+1.*0.2261725957012447*sin(\t r)});
\draw [shift={(5.337200378913295,6.745049224444948)},line width=0.4pt]  plot[domain=4.9134570762437075:5.832726076495774,variable=\t]({1.*1.1133137627604077*cos(\t r)+0.*1.1133137627604077*sin(\t r)},{0.*1.1133137627604077*cos(\t r)+1.*1.1133137627604077*sin(\t r)});
\draw [shift={(4.642083259921757,9.789377117550492)},line width=0.4pt]  plot[domain=4.827682869882762:4.930718586642295,variable=\t]({1.*4.235767135465101*cos(\t r)+0.*4.235767135465101*sin(\t r)},{0.*4.235767135465101*cos(\t r)+1.*4.235767135465101*sin(\t r)});
\draw [line width=0.8pt,color=ffqqqq] (4.755638769444044,7.073265264839444)-- (4.755638769444044,5.700383831150284);
\draw [line width=0.8pt,color=ffqqqq] (4.755638769444044,5.521292638718075)-- (4.755638769444044,4.605041317096833);
\draw [shift={(12.129936838883907,6.745049224444946)},line width=0.4pt]  plot[domain=4.913457076243709:5.832726076495776,variable=\t]({1.*1.113313762760407*cos(\t r)+0.*1.113313762760407*sin(\t r)},{0.*1.113313762760407*cos(\t r)+1.*1.113313762760407*sin(\t r)});
\draw [shift={(11.43481971989235,9.789377117550597)},line width=0.4pt]  plot[domain=4.827682869882763:4.9307185866422945,variable=\t]({1.*4.235767135465208*cos(\t r)+0.*4.235767135465208*sin(\t r)},{0.*4.235767135465208*cos(\t r)+1.*4.235767135465208*sin(\t r)});
\draw [shift={(11.369108705107225,5.8791184149562525)},line width=0.4pt]  plot[domain=5.0417128980032055:6.228110483937385,variable=\t]({1.*0.5543125119013331*cos(\t r)+0.*0.5543125119013331*sin(\t r)},{0.*0.5543125119013331*cos(\t r)+1.*0.5543125119013331*sin(\t r)});
\draw [shift={(11.726885318762076,5.812144430168311)},line width=0.4pt]  plot[domain=0.1842017557222193:2.1870758851406222,variable=\t]({1.*0.19906302498267991*cos(\t r)+0.*0.19906302498267991*sin(\t r)},{0.*0.19906302498267991*cos(\t r)+1.*0.19906302498267991*sin(\t r)});
\draw [shift={(11.676691302785517,5.775748061316204)},line width=0.4pt]  plot[domain=2.7340011295637723:4.998612206463343,variable=\t]({1.*0.2261725957012434*cos(\t r)+0.*0.2261725957012434*sin(\t r)},{0.*0.2261725957012434*cos(\t r)+1.*0.2261725957012434*sin(\t r)});
\draw [line width=0.8pt,color=ffqqqq] (11.548375229414656,7.073265264839444)-- (11.548375229414656,5.700383831150284);
\draw [line width=0.8pt,color=ffqqqq] (11.548375229414656,5.521292638718075)-- (11.548375229414656,3.4844979297781022);
\draw (2.5,3.32) node[anchor=north west] {$A=(c,\ \ell_1\cup\ell_2, \ p)$};
\draw (8.9,3.2) node[anchor=north west] {$B=(c, \ \ell_1\cup \ell_2, \ \pi)$};
\draw (2.5,7.35) node[anchor=north west] {$\ell_1$};
\draw (4.2,7.35) node[anchor=north west] {$\ell_2$};
\draw (9.2,7.25) node[anchor=north west] {$\ell_1$};
\draw (10.9,7.25) node[anchor=north west] {$\ell_2$};
\draw (2.1012323728803306,4.6) node[anchor=north west] {$p$};
\draw (9,3.6) node[anchor=north west] {$\pi$};
\draw (3.6,6.95) node[anchor=north west] {$c$};
\draw (10.2,6.95) node[anchor=north west] {$c$}; 
\begin{scriptsize}
\draw [fill=black] (3.016984405628848,6.2455180002658315) circle (1.0pt);
\draw [fill=black] (4.755638769444044,5.354593988343822) circle (1.0pt);
\draw [fill=black] (9.809720865599461,6.2455180002658315) circle (1.0pt);
\draw [fill=black] (11.548375229414656,5.354593988343822) circle (1.0pt);
\end{scriptsize}
\end{tikzpicture}
\caption{$A$ and $B$ are not isotopic in space.}
\label{figure_rail_arcs}
\end{figure}

In the knotoid method, the  invariants applied to rail isotopy come from the invariants of knotoids and many of these are produced by the over or under completion of a knotoid diagram to a knot by connecting the head and leg of the knotoid by an arc totally on the top or totally in the bottom of the knotoid diagram \cite{GK,Tu}. Our over and under companion loops (see Figure \ref{figure_companion_loops}) for rail knotoids correspond to the the over and under closure for knotoids. 

Let us also note that the obvious way to correspond a rail knotoid diagram to a knotoid diagram by just forgetting the rails, is not a proper way to correspond equivalent classes of such diagrams (rail knotoids and knotoids), since for example the rail knotoid of the diagram in Figure \ref{figure_corespondance} is non trivial (its over companion loop is non trivial) but the corresponding knotoid diagram when we forget the rails represents the trivial knotoid.

\begin{figure}
	\centering
	\definecolor{ffqqqq}{rgb}{1.,0.,0.}
\begin{tikzpicture}[line cap=round,line join=round,>=triangle 45,x=1.0cm,y=1.0cm]
\clip(1.7278744543094895,3.6525169540970137) rectangle (9.436992819629301,6.20458945578052);
\draw [shift={(1.9957699809723821,5.155855965913496)},line width=0.4pt]  plot[domain=-1.782397145562955:1.5520165821775895,variable=\t]({1.*0.5311544590003275*cos(\t r)+0.*0.5311544590003275*sin(\t r)},{0.*0.5311544590003275*cos(\t r)+1.*0.5311544590003275*sin(\t r)});
\draw [shift={(2.3327131329134763,5.335347187679804)},line width=0.4pt]  plot[domain=2.662674876310721:4.496762038732055,variable=\t]({1.*0.5933848566941505*cos(\t r)+0.*0.5933848566941505*sin(\t r)},{0.*0.5933848566941505*cos(\t r)+1.*0.5933848566941505*sin(\t r)});
\draw [line width=0.8pt,color=ffqqqq] (2.9866668619088794,6.120953278435534)-- (2.9866668619088794,3.7511139107414495);
\draw [shift={(2.8338050929288503,6.42598234302416)},line width=0.4pt]  plot[domain=4.4739275061712265:4.79854355839295,variable=\t]({1.*1.7764701770836266*cos(\t r)+0.*1.7764701770836266*sin(\t r)},{0.*1.7764701770836266*cos(\t r)+1.*1.7764701770836266*sin(\t r)});
\draw [line width=0.8pt,color=ffqqqq] (1.8842141157654901,6.120953278435534)-- (1.8842141157654901,5.166072947130226);
\draw [line width=0.8pt,color=ffqqqq] (1.8842141157654901,4.836205196315664)-- (1.8842141157654901,3.7511139107414495);
\draw [shift={(5.02121667460224,5.155855965913496)},line width=0.4pt]  plot[domain=-1.7823971455629515:1.5520165821775869,variable=\t]({1.*0.5311544590003271*cos(\t r)+0.*0.5311544590003271*sin(\t r)},{0.*0.5311544590003271*cos(\t r)+1.*0.5311544590003271*sin(\t r)});
\draw [shift={(5.358159826543334,5.335347187679803)},line width=0.4pt]  plot[domain=2.6626748763107186:4.496762038732058,variable=\t]({1.*0.5933848566941501*cos(\t r)+0.*0.5933848566941501*sin(\t r)},{0.*0.5933848566941501*cos(\t r)+1.*0.5933848566941501*sin(\t r)});
\draw [shift={(5.859251786558709,6.425982343024171)},line width=0.4pt]  plot[domain=4.473927506171227:4.79854355839295,variable=\t]({1.*1.7764701770836382*cos(\t r)+0.*1.7764701770836382*sin(\t r)},{0.*1.7764701770836382*cos(\t r)+1.*1.7764701770836382*sin(\t r)});
\draw [shift={(8.340364566830916,5.155855965913496)},line width=0.4pt]  plot[domain=-1.7823971455629586:1.5520165821775938,variable=\t]({1.*0.5311544590003279*cos(\t r)+0.*0.5311544590003279*sin(\t r)},{0.*0.5311544590003279*cos(\t r)+1.*0.5311544590003279*sin(\t r)});
\draw [shift={(8.677307718772006,5.335347187679803)},line width=0.4pt]  plot[domain=2.6626748763107186:4.496762038732056,variable=\t]({1.*0.5933848566941501*cos(\t r)+0.*0.5933848566941501*sin(\t r)},{0.*0.5933848566941501*cos(\t r)+1.*0.5933848566941501*sin(\t r)});
\draw [shift={(9.178399678787382,6.425982343024167)},line width=0.4pt]  plot[domain=4.4739275061712265:4.798543558392949,variable=\t]({1.*1.776470177083634*cos(\t r)+0.*1.776470177083634*sin(\t r)},{0.*1.776470177083634*cos(\t r)+1.*1.776470177083634*sin(\t r)});
\draw [line width=0.4pt] (4.9096608093953495,5.166072947130226)-- (4.9096608093953495,6.);
\draw [line width=0.4pt] (4.9096608093953495,6.)-- (6.012113555538739,6.);
\draw [line width=0.4pt] (6.012113555538739,6.)-- (6.012113555538738,4.656101112644045);
\draw [line width=0.4pt] (4.9096608093953495,4.836205196315664)-- (4.9096608093953495,4.6365483997700085);
\draw [shift={(8.320040879231563,5.47016480085748)},line width=0.4pt]  plot[domain=1.431914094995293:2.455654002194643,variable=\t]({1.*0.2188592817991089*cos(\t r)+0.*0.2188592817991089*sin(\t r)},{0.*0.2188592817991089*cos(\t r)+1.*0.2188592817991089*sin(\t r)});
\draw (2.3175963869694165,5.808052294164358) node[anchor=north west] {$K$};
\draw (5.35,5.838555152750216) node[anchor=north west] {$K_o$};
\draw (8.7,5.802968484400048) node[anchor=north west] {$K'$};
\begin{scriptsize}
\draw [fill=black] (1.8842141157654901,4.6365483997700085) circle (1.0pt);
\draw [fill=black] (2.9866668619088794,4.656101112644045) circle (1.0pt);
\draw [fill=black] (4.9096608093953495,4.6365483997700085) circle (1.0pt);
\draw [fill=black] (6.012113555538738,4.656101112644045) circle (1.0pt);
\draw [fill=black] (8.228808701624022,4.6365483997700085) circle (1.0pt);
\draw [fill=black] (9.33126144776741,4.656101112644045) circle (1.0pt);
\end{scriptsize}
\end{tikzpicture}
	\caption{For the pictured rail knotoid diagram $K$, the  over companion loop $K_o$ is non-trivial, but the corresponding knotoid diagram $K'$ obtained by forgetting the rails is trivial.}
\label{figure_corespondance}
\end{figure}
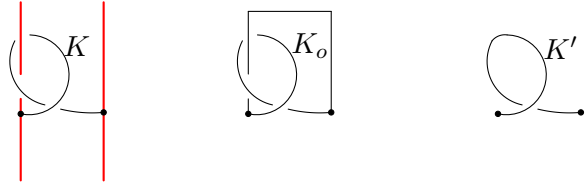

It is reasonable to expect that a better understanding of the similarities and differences between knotoids and rail knotoids will be proved fruitful.

\end{document}